%% file: BS1k.tex
\documentclass[12pt]{article}
\usepackage{amsmath,amsfonts,amsthm,amssymb}
\usepackage{xspace}
\usepackage{enumitem}
\usepackage{comment}
\usepackage{graphicx}

\newcommand{\xtriarrow}[1]{\blacktriangleright\mkern-8mu\xrightarrow{#1}}
\newcommand{\xdiarrow}[1]{\rotatebox[origin=c]{45}{\scalebox{0.75}{$\square$}}\mkern-8mu\xrightarrow{#1}}
\newcommand{\xdibarrow}[1]{\rotatebox[origin=c]{45}{\scalebox{0.75}{$\blacksquare$}}\mkern-8mu\xrightarrow{#1}}
\newcommand{\xtribarrow}[1]{\rhd\mkern-8mu\xrightarrow{#1}}

\newcommand{\Q}{{\mathbb Q}}
\newcommand{\Z}{{\mathbb Z}}
\newcommand{\N}{{\mathbb N}}

\newcommand{\BS}{\mathsf{BS}}
\newcommand{\NFB}{\mathsf{NF}}
\newcommand{\NFFB}{\mathsf{NF^{\rm frac}}}
\newcommand{\fw}{\hat{w}}
\newcommand{\fx}{\hat{x}}
\newcommand{\fy}{\hat{y}}
\newcommand{\fz}{\hat{z}}
\newcommand{\up}{{+}}
\newcommand{\um}{{-}}
\newcommand{\ud}{{\centerdot}}

\newcommand{\LF}{\mathsf{L_{\rm frac}}}
\newcommand{\Con}{\mathsf{Con}}
\newcommand{\emptystring}{\epsilon}
\newcommand{\ainv}{{\hbox{\sc a}}}
\newcommand{\binv}{{\hbox{\sc b}}}

\newcommand{\aA}{\alpha}
\newcommand{\bB}{\beta}

\newcommand{\sct}{{\hbox{\sc t}}}
\newcommand{\scr}{{\hbox{\sc r}}}
\newcommand{\OL}{{\rm \sc 0l}\xspace}
\newcommand{\E}{{\rm \sc edt0l}\xspace}
\newcommand{\NE}{{\rm \sc et0l}\xspace}
\newcommand{\cG}{\mathcal{G}}
\newcommand{\cH}{\mathcal{H}}

\newcommand{\cP}{\mathcal{P}}
\newcommand{\cR}{\mathcal{R}}
\newcommand{\cT}{\mathcal{T}}
\newcommand{\cU}{\mathcal{U}}
\newcommand{\cV}{\mathcal{V}}
\newcommand{\sa}{\mathsf{a}}
\newcommand{\sA}{\mathsf{A}}

\newcommand{\ct}{\mathsf{c}}

\newtheorem{theorem}{Theorem} 
\newtheorem{lemma}[theorem]{Lemma} 
\newtheorem{proposition}[theorem]{Proposition}

\newtheorem{definition}[theorem]{Definition}

\newenvironment{proofof}[1]{\normalsize {\it Proof of #1}.}{{\hfill $\Box$}}
\newenvironment{mylist}{\begin{list}{}{
\setlength{\parskip}{0mm}
\setlength{\topsep}{2mm}
\setlength{\parsep}{0mm}
\setlength{\itemsep}{0.5mm}
\setlength{\labelwidth}{7mm}
\setlength{\labelsep}{3mm}
\setlength{\itemindent}{0mm}
\setlength{\leftmargin}{12mm}
\setlength{\listparindent}{6mm}
}}{\end{list}}
\newenvironment{bracedeqn}
  {\left.\begin{aligned}}
  {\end{aligned}\right\rbrace}
\title{Using \E systems to solve some equations in the solvable Baumslag-Solitar
groups}
\author{Andrew Duncan, Alex Evetts, Derek F. Holt, Sarah Rees}
\begin{document}
\maketitle

\begin{abstract}
	We investigate the solution sets to equations in the solvable Baumslag-Solitar groups \(BS(1,k)\), \(k\geq2\), and show that these sets are represented by \E languages in some cases. In particular, we prove that the multiplication table of such a group forms an \E language with respect to a specific natural
normal form for group elements. 
\end{abstract}

\section{Introduction}
\label{sec:intro}
This work was motivated in part by Ciobanu, Diekert and Elder's proof \cite{CDE} that
the solution sets of systems of equations over free groups are \E (which was subsequently generalised to virtually free groups by Diekert and Elder, hyperbolic groups by Ciobanu and Elder \cite{CE}, and right-angled Artin groups by Diekert, Je\.z and Kufleitner \cite{DJK}). Evetts and Levine \cite{EL} proved that the same thing is true in virtually abelian groups. 
Further motivation was provided by 
the proof of Kharlampovich, López and Myasnikov \cite{KLM}
that it is decidable whether or not an equation over a group $G$ is solvable,
for all $G$ in a family of groups that includes the solvable Baumslag-Solitar
groups and groups with structure $A \wr \Z$ for $A$ finitely generated abelian 
(such as the lamplighter group when $|A|=2$).
We wanted to understand for which types of groups solution sets to
systems of equations might be \E languages, and why this might be a natural
family of languages in which to find such solution sets.

In this article we have attempted to provide an accessible description of this relatively unknown family of languages, and of how we may find solutions for group equatons within it.
Using the solvable Baumslag-Solitar groups 
as our `testbed', we have examined various rather elementary equations over
these groups, proved some to have solution sets that are \E languages,
and provided examples of others that seem not to be (although we have not yet proved conclusively that they are not).

The following Section~\ref{sec:esystems} contains the definitions of \NE and \E systems and related languages, relating them to the better known families of 
context-free and indexed languages, and listing some of the operations under which the sets of \NE and \E languages are closed.
The Baumslag-Solitar groups are studied from Section~\ref{sec:BSdefsNF} onwards,
with definitions and the construction of two different normal forms in Section~\ref{sec:BSdefsNF}, preliminary results in Section~\ref{sec:BSprelim} and the
consideration of 
particular equations relating to centralisers, conjugacy, multiplication 
and inversion, in Sections~\ref{sec:BScentConj}--\ref{sec:BSexplicit}.
In Section~\ref{sec:BSnonEg} we describe a solution set for a centraliser 
equation that we believe is not \E, and explain why we believe this, 
while not giving a proof of that fact.

\section{\NE and \E systems}
\NE languages, introduced by Rozenberg  \cite{Rozenberg}, generalise both context-free and \OL languages, and are most naturally defined through their grammars, known as \NE systems. An important component of such a grammar is a set of tables. 

\label{sec:esystems}
\begin{definition}
Let $\cV$ be a finite alphabet. A \emph{table} for $\cV$ is a finite subset
of $\cV\times \cV^*$, considered as a finite collection of rewriting
rules of the form $v \to w_1,\ldots,w_r$, for $v\in \cV$, $w_i\in \cV^*$,
whose application replaces each instance of $v$ with any one of $w_1$,$\cdots$, $w_r$.
\end{definition}

We use the conventions that (1) when a table is applied to a word it must be applied to every letter within that word and (2) if a table does not specify a rewrite for some
letter in $\cV$, then applying the table fixes that letter.
A table will act on the right of a word, so we write $w\sct$ for the result of 
applying a table $\sct$ to a word $w$. Applying one table after another will be
denoted by concatenation; note that when applying a string of tables that string should be read from left to right.
Note that if the right hand side of each rule in a table contains only one 
word, then that table defines a free monoid endomorphism of $\cV^*$.

\begin{definition}
	An \NE system is a tuple $\cH=(\cV,\cU,\cR,v_0)$ where
	\begin{enumerate}
	\item $\cV$ is a finite alphabet,
	\item $\cU\subset \cV$ is the set of terminals,
	\item $\cR$ is a regular subset of $\cT^*$ for some finite set of 
	tables $\cT$ for $\cV$, called the \emph{rational control} of $\cH$.
	\item $v_0\in \cV^*$ is a chosen word called the \emph{start word} or \emph{axiom}.
	\end{enumerate}
The language $\{v\in \cU^*\colon v=v_0\scr\text{ for some }\scr\in \cR\}$ is the
language of the system $\cH$. A language arising from an \NE system is called
an \NE language.
\end{definition}
It is proved in \cite{Culik} and \cite{EFS} that the class of \NE languages
is properly contained in the class of indexed languages, as defined in
\cite{Aho}. 

\begin{definition}
	An \E system is an \NE system $\cH=(\cV,\cU,\cR,v_0)$ where each table in $\cT$,
	the alphabet of $\cR$, is a free monoid endomorphism $\cV^*\to \cV^*$;
	that is, each rule in a table contains a unique word on the right hand side.
\end{definition}

Let $\cG = (\cV,\cU,\cP,s_0)$ be a context-free grammar, with
set $\cV$ of terminals, $\cU \subseteq \cV$, $\cP$ of productions, and start
variable $s_0$; for such a grammar the left hand side of any production is necessarily in $\cV \setminus \cU$,
then we can form an \NE system $(\cV,\cU,\cR,s_0)$ with the same language as $\cG$ as follows:

For each $v \in \cV$, we define $\cP_v$ to be the set of all productions
in $\cP$ with left hand side $v$.
If the productions in $\cP_v$ are
$v\rightarrow u_1,\ldots, v \rightarrow u_k$, then we define $r_v$ to be the
rule $v\rightarrow u_1, u_2, \cdots u_k,v$.
Note that the fact that $v$ is on the right-hand side of the rule indicates that
we are not obliged to apply a non-trivial production to $v$.

We define the table $\sct_v$ to be the singleton set $\{ r_v\}$,
$\cT$ to be $\{ \sct_v : v \in \cV \}$,
and then $\cR$ to be $\cT^*$.
The language of the \NE system $(\cV,\cU,\cR,s_0)$ is the context-free language
generated by the grammar $\cG$.
But note that context-free languages exist that do not arise as the languages
of \E systems \cite{EF}.

The following two lemmas are standard (see \cite{RS}).
\begin{lemma}
\label{lem:regedt0l}
	The class of \E languages contains the class of regular languages.
\end{lemma}
\begin{lemma}
\label{lem:closure}
	The classes of \E and \NE languages are both closed under the following operations:
	\begin{enumerate}
		\item finite union,
		\item intersection with regular languages,
		\item concatenation,
		\item Kleene star, 
		\item image under free monoid homomorphisms.
	\end{enumerate}
	The class of \NE languages is additionally closed under taking preimages
	of free monoid homomorphims.
\end{lemma}

\section{The solvable Baumslag-Solitar groups: definition and normal forms}
\label{sec:BSdefsNF}
We use the notation $\N := \{ n: n \in \Z, n >0 \}$, $\N_0 := \N \cup \{0\}$.

The solvable Baumslag-Solitar groups are the groups defined by the
presentations

$$\BS(1,k) = \langle a,b \mid b^{-1} a b = a^k \rangle$$
for some fixed $k \in \Z \setminus\{0\}$. We shall assume here that $k>0$, since
the groups with $k<0$ are similar, but with additional minor complications.
Although we have not checked all of the details, we believe that all of
the results proved in this section remain true when $k$ is negative, but
their proofs require subdivisions into more cases than when $k>0$.
We shall assume further that $k>1$, since $\BS(1,1)$ is free abelian.
For ease of notation, we shall abbreviate $\BS(1,k)$ as $\BS$.

Note that $\BS$ is a split extension of the infinite abelian group
$N = \langle a \rangle^\BS$ (the normal closure of the subgroup
$\langle a \rangle$ in $\BS$) by the infinite cyclic group $\langle b \rangle$.
We denote by $a^{\frac{1}{k}}$ the element $bab^{-1}$ and note that 
$(a^{\frac{1}{k}})^k=a$. Similarly we represent 
$b^ja^ib^{-j}$ by $a^{\frac{i}{k^j}}$, where $i,j\in \Z$, $i \neq 0$, $j>0$.
We also write $\ainv$ for $a^{-1}$, and $\binv$ for $b^{-1}$.

With this notation, we observe that any element $g$ of $\BS$ can be represented
as a product $b^r \aA ^u$ for $\aA \in \{a,\ainv\}$, where $r \in \Z$ and 
\[u =  \left(\frac{i_m}{k^m} + \frac{i_{m-1}}{k^{m-1}} + \cdots
+ \frac{i_1}{k}\right) + s,\quad 
m,s \in \N_0,\,i_j \in \{0,\ldots,k-1\},\]
and we interpret $\aA^u$ as the product
\begin{eqnarray*}
&& (b^m\aA b^{-m})^{i_m}  (b^{m-1}\aA b^{-m+1})^{i_{m-1}}  \cdots
(b\aA b^{-1})^{i_1}\aA ^s.
\end{eqnarray*}
We require that
either $m=0$ (that is, the sequence $i_1,\ldots,i_m$ is empty)
or $m>0$ and $i_m \neq 0$.
We shall call $b^r\aA^u$  the {\em fractional representation} of the
group element, and we shall call $s$ and
$i_m/k^m + i_{m-1}/k^{m-1} + \cdots + i_1/k$ the  {\em integral}
and {\em fractional} parts of the exponent $u$.

The fractional representation  
provides a mechanism for representing each element of $\BS$ by a unique string
of symbols from the alphabet
$\{b,\binv,0,1,\ldots,k\!-\!1,\ud,+,-\}$.
The group element
$b^r\aA^u$ is represented by
the string $\fw$ given by
\begin{equation*}
\begin{bracedeqn}
&\bB \cdots \bB + i_m i_{m-1} \cdots i_1 \ud s_0 s_1 \cdots s_p\,{\rm when}\ \aA = a,\\
&\bB\cdots \bB - i_m i_{m-1} \cdots i_1 \ud s_0 s_1 \cdots s_p \, {\rm when}\ \aA = \ainv,\\
\hbox{\rm where}\ &
p \ge -1,\,i_1,\ldots,i_m,s_0,\ldots s_p \in \{0,1,\ldots,k\!-\!1\},\\
\hbox{\rm and if}\ p\neq -1,\hbox{\rm then}\ 
&s_p \ne 0\,{\rm and}\, s = s_0 + s_1k + s_2k^2 + \cdots + s_pk^p,
\end{bracedeqn}
\shoveright \text{(1)}
\end{equation*}
where $\bB$ is $b$ when $r\geq 0$, $\binv$ when $r<0$, 
and where the prefix $\bB\cdots \bB$ consists of $|r|$ copies of $\bB$
(and in future will generally be abbreviated as $\beta^r$).
By convention, we define $p=-1$ if the sequence $s_0,\ldots,s_p$ is empty
and $s=0$.

Then the suffix ${\pm} i_mi_{m-1} \cdots i_1 \ud  s_0 \cdots s_p$ of $\fw$
is precisely the representation of $u$ in base $k$ written backwards.
The symbol $\ud$ (which would be the decimal point when $k=10$) is known
as the \emph{radix point}.
We shall call the subwords $i_mi_{m-1} \cdots i_1$ and $s_0 \cdots s_p$
the \emph{fractional} and \emph{integral} parts of $\fw$.
Note that either or both of these could be empty, and the representation of
the element $\beta^t$ is $\beta^t+\ud$ for $t \ge 0$.

We denote the representation of $g \in \BS$ defined in (1) by $\NFFB(g)$,
and define $\NFFB:=\{\NFFB(g):g \in \BS\}$. 
We can think of $\NFFB$ as a normal form for $\BS$, although of course the elements of its alphabet are not all within the group.

Now the product
\[ b^r(b^m\aA b^{-m})^{i_m}  (b^{m-1}\aA b^{-m+1})^{i_{m-1}}  \cdots
(b\aA b^{-1})^{i_1}\aA ^s 
\]
that represents $b^r\alpha^u$ can be freely reduced to give the word
\[w= b^{t} \aA^{i_m} \binv \aA^{i_{m-1}} \binv \cdots \aA^{i_1} \binv \aA^s,
   \eqno{(2)} \] where $t=r+m$.

When $t\geq 0$,
we shall refer to the subwords $b^t$, $\aA^{i_m} \binv \aA^{i_{m-1}} \binv \cdots
\aA^{i_1}\binv$, and $\aA ^s$  of $w$ as its {\em left, central and right} subwords,
respectively, 
but if $t<0$ we define the left subword to be empty,
the central subword to be
 $\binv^{|t|}\aA^{i_m} \binv \aA^{i_{m-1}} \binv \cdots\aA^{i_1}\binv$
 and the right subword to be $\aA^s$.
We call the subwords $\aA^{i_j}\binv$ 
together with the letters $\binv$ in a $\binv^{|t|}$ prefix
the {\em components} of the central subword.

Observe that the integral part $s$ of $\fw$ is the base $k$ representation of
the exponent of $\aA$ in the right subword of $w$, whereas the fractional part
of $\fw$, together with the value of $t$ when $t$ is negative,
determines the central subword of $w$.

We denote the representation of $g \in \BS$ defined in (2) by $\NFB(g)$,
and put $\NFB:=\{\NFB(g):g \in \BS\}$. 
Note that both $\NFB$ and $\NFFB$ are regular languages.

\section{Preliminary results}
\label{sec:BSprelim}
\begin{proposition}
\label{prop:fixedr}
Fix a constant $r\in\Z$.
The subset $\NFB_r= \{ \NFB(g) : g = b^r \alpha^u \mbox{ for some }u \}$
of words in $\NFB$ corresponding to this value of $r$ forms an \E language.
\end{proposition}
\begin{proof}
	Let  $\cU=\{a,\ainv,b,\binv\}$ and $\cV=\{a,\ainv,b,\binv,S,T\}$, and define a set $\cT$
	of endomorphisms of $\cV^*$  (each of which is the sole entry of a table within $\cT$)
	as follows:
	\begin{align*}
		\phi_{aj}&\colon S\mapsto Sa^j\binv,\text{ for each }0\leq j\leq k-1 \\
		\phi_{\ainv j}&\colon S\mapsto S\ainv^j\binv,\text{ for each }0\leq j\leq k-1 \\
		\psi_{aj}&\colon S\mapsto bSa^j\binv,\text{ for each }0\leq j\leq k-1 \\
		\psi_{\ainv j}&\colon S\mapsto bS\ainv^j\binv,\text{ for each }0\leq j\leq k-1 \\
		\theta&\colon S\mapsto bS \\
		\mu_a&\colon T\mapsto Ta \\
		\mu_\ainv&\colon T\mapsto T\ainv \\
		\nu&\colon S\mapsto\emptystring, T\mapsto\emptystring.
	\end{align*}
	We define 
	\begin{eqnarray*}
		\Phi_a := \{ \phi_{a0},\ldots,\phi_{a,k-1}\},&&
		\Phi_A := \{ \phi_{A0},\ldots,\phi_{A,k-1}\},\\
		\Psi_a := \{ \psi_{a0},\ldots,\psi_{a,k-1}\},&&
		\Psi_A := \{ \psi_{A0},\ldots,\psi_{A,k-1}\}.
	\end{eqnarray*}

	First, suppose that $r\geq0$. In this case, the language in question is
	\[\{b^rb^m\aA^{i_m}\binv\aA^{i_{m-1}}\binv\cdots \aA^{i_1}\binv\aA^s\colon  m\geq0, s\geq0\}.\]
	This is the language of the \E system $(\cV,\cU,\cR_1,ST)$ with rational
	control $\cR_1$ given by $\cR_1 := \cR_{1a} \mid \cR_{1\ainv}$
(recall that the symbol `$\mid$' denotes union in the standard notation for
regular sets) where
	\[\cR_{1a}=\theta^r \left (\Psi_a^* \setminus \Psi_a^*\psi_{a0}\right)\mu_a^*\nu,\]
and $\cR_{1\ainv}$ is defined similarly with $\ainv$ in place of $a$.

	Now suppose that $r<0$. If we also have $t=m+r<0$, then the left
subword is empty and there are only finitely many possibilities for the
central subword, so the language in this case regular. For the case
$t \ge 0$, we need to show that the following language is \E:
	\[\{b^t\aA^{i_{t+|r|}}\binv\aA^{i_{t+|r|-1}}\binv\cdots \aA^{i_{t+1}}\binv\aA^{i_t}\binv\cdots \aA^{i_{1}}\binv\aA^s\colon t\geq0, s\geq0\}.\]
	This is the language of the \E system $(\cV,\cU,\cR_2,ST)$ with rational 
	control $\cR_2$ given by  $\cR_2 :=  \cR_{2a} \mid  \cR_{2\ainv}$, where
	\[\cR_{2a}:= \Phi_a^{|r|}\left (\Psi_a^* \setminus \Psi_a^*\psi_{a0}\right)\mu_a^*\nu, \]
and $\cR_{2\ainv}$ is defined similarly with $\ainv$ in place of $a$.
\end{proof}

The following lemma will be used several times in the proofs in
Section~\ref{sec:BSmultInv} below.

\begin{lemma}
\label{lem:expedt0l}
Suppose that $r>0$. Fix constants $n_0\in\N$, $\lambda,c\in\Z$,
let $\aA$ be equal to either $a$ or to $\ainv$, and let
  $w\in (\aA \binv^+\mid \aA^2\binv^+\mid\cdots\mid \aA^{k-1}\binv^+)^*$ be a constant word (with \(k\) fixed as above).
   Let $\{s_n\}_{n\in\N}$ be a fixed sequence of integers
such that, for each $n\geq n_0$, we have
 \[s_{n+1}=k^{r}s_n+\lambda\quad\hbox{\rm and}\quad s_{n+1} > s_n.\]
	Then, for some $n' \in \N$, the following subset of
	$\{a,\ainv,b,\binv\}^*$ 
	is \E:
	\[\{b^{rn+c} w \aA^{s_n} \colon n\geq n'\}.\]
\end{lemma}
\begin{proof}
We prove this for $\aA = a$; the case $\aA = \ainv$ is similar.
	First, consider the case where \(\lambda\geq0\).
	Define endomorphisms $\phi,\psi$ of $\{a,\ainv,b,\binv,S\}^*$ as follows:
	\begin{align*}
		\phi&\colon S\mapsto b^rSa^\lambda,~a\mapsto a^{k^r} \\
		\psi&\colon S\mapsto w.
	\end{align*}
	Choose $n'\geq n_0$ large enough so that $rn'+c\geq0$. The \E system with axiom $b^{rn'+c}Sa^{s_{n'}}$
	and rational control $\cR=\phi^*\psi$ gives the language $\{b^{rn+c}w a^{s_n}\colon n\geq n'\}$.
	
	Now suppose that \(\lambda<0\). Choose \(n'\geq n_0\) large enough so that \(rn'+c\ge 0\) and so that \(s_{n'}>|\lambda|\). Consider an extended alphabet \(\{a,\ainv,\sa,b,\binv,S\}\), endomorphisms
	\begin{align*}
		\phi\colon & S\mapsto b^rS, a\mapsto a^{k^r}, \sa\mapsto a^{k^r-2}\sa \\
		\psi\colon & S\mapsto w, \sa\mapsto a,
	\end{align*}
	axiom \(b^{rn'+c}Sa^{s_{n'}-|\lambda|}\sa^{|\lambda|}\), and rational control \(\phi^*\psi\). We claim that this \E system produces the language in question. To see this, consider an application of \(\phi\):  \[b^{rn+c}Sa^{s_n-|\lambda|}\sa^{|\lambda|}\mapsto b^{r(n+1)+c} S a^{k^r(s_n-|\lambda|)}(a^{k^r-2}\sa)^{|\lambda|}.\]
	At each such an application, we increase the number of \(a\)'s from \(s_n-|\lambda|\) to \(k^r(s_n-|\lambda|) + (k^r-2)|\lambda| = k^rs_n - 2|\lambda| = s_{n+1}-|\lambda|\), and keep the number of \(\sa\)'s fixed at \(|\lambda|\).
	If this were the final application of \(\phi\), we would then apply \(\psi\) to convert the \(\sa\)'s to \(a\)'s, and insert \(w\):
	\[b^{r(n+1)+c} S a^{k^r(s_n-|\lambda|)}(a^{k^r-2}\sa)^{|\lambda|} \mapsto b^{r(n+1)+c}w a^{k^rs_n -|\lambda|}=b^{r(n+1)+c}wa^{s_{n+1}}.\]
\end{proof}

\section{Centralisers and conjugacy}
\label{sec:BScentConj}
\begin{proposition}
\label{prop:BScent}
Centralisers of elements in $\BS(1,k)$ are \E; that is, for fixed $g \in \BS$,
the set
$C_g := \{  w \in \NFB : w g =_\BS g w \}$ is \E.
\end{proposition}
\begin{proof}

This follows from Lemma~\ref{lem:regedt0l} when $g = 1$, because then
$C_g = \NFB$, which is a regular language.

If  $1 \ne g \in N$, then $C_g = N$, and $C_g$ consists of those words in $\NFB$
with $t=m$, which is \E by Proposition~\ref{prop:fixedr}.

If $g \not\in N$, then $C_N(g) = 1$, so $C_\BS(g)$ is an infinite cyclic group
$\langle g'\rangle$, for which $g' \not \in N$ and $C_\BS(g)=C_\BS(g')$.
In this case we replace $g$ by $g'$, for ease of notation.

So we need to prove that $\{ \NFB(g^n) : n \in \Z \}$ is \E and, since this is
a union of the positive and negative powers of $g$, it is enough
(by Lemma~\ref{lem:closure}) to prove that $\{ \NFB(g^n) : n \ge 0 \}$ is \E.

We use the notation introduced above for $w := \NFB(g)$, and we have
$w = b^{t'} w_c \aA^s$, where $b^{t'}$, $w_c$ and $\aA^s$ are its left,
central and right subwords, with $\aA = a$ or $\ainv$.
(So $t'=t$ when $t \ge 0$ and $t'=0$ when $t<0$.)
Note that the fractional representation of $w$ is of the form $b^r\aA^u$
with $r:=t-m$ and $0 \le u \in \Q$, and so that of $g^n$ is
$b^{rn}\alpha^{u_n}$ with $u_n = u(1+k^{r} + \cdots + k^{(n-1)r})$.
We consider two cases.

{\bf Case 1}: $t < m$ (so $w$ has more occurrences of $\binv$ than of $b$).
So $t-m < 0$.
Since the power of $\binv$ in the fractional representation of $g^n$ is
$\binv^{n(m-t)}$, the central subword of $\NFB(g^n)$ must have at least
$n(m-t)$ components and so there exists $n' \ge 0$ such that this
number of components is at least $t'$ for all $n \ge n'$.

For $n \ge n'$, let $\NFB(g^n) = b^{t'_n}w_{cn} \aA ^{s_n}$, where
$w_{cn}$ is its central subword, and let  $w_{cn} = \gamma_n\beta_n$, where
$\beta_n$ consists of its final $t'$ components.
Note that the sequence $u_n$ is bounded above by $u/(1-k^r)$ in this
case, and so $s_n$ is also bounded above.
We claim  that $\NFB(\beta_n \aA^{s_n} w)$ has empty left subword and that
its central subword has exactly $m$ components. From this it follows that
$\NFB(g^{n+1}) = b^{t'_n} \gamma_n \NFB(\beta_n \aA^{s_n} w)$, so
$w_{c,n+1}$ has $m-t'$ more components than $w_{cn}$. 
  
To prove the claim, note that, since $\beta_n$ has exactly $t'$ components,
$\beta_n \aA^{s_n} b^{t'}$ is equal in $\BS$ to some power  $\aA^{x_n}$ of $\aA$
 with $x_n \in \Z$ and $x_n \ge 0$. Then, when we put
$\beta_n \aA^{s_n} w =_\BS \aA^{x_n} w_c \aA^s$ into normal form, we may move
some of the $m$ letters $\binv$ in $w_c$ to the left, but we end up with a central
subword consisting of $m$ components followed by a power of $\aA$
(in fact $\aA^{s_{n+1}}$).

Now, by considering the fractional representation of $g^n$ described above
and noting that $r<0$ in this case, we see that $s_n$, which is the integral
part of $u_n$, is a non-decreasing sequence which (as we observed earlier) is
bounded above, and so $s_n$ must be constant for sufficiently large $n$.
Also, $\beta_n$ and hence also the suffix $\beta_n \aA^{s_n}$ of
$\NFB(g^n)$ must eventually repeat, and it follows easily that
the set $\{ \NFB(g^n) : n \in \Z \}$ is regular, and hence \E by
Lemma~\ref{lem:regedt0l}.

{\bf Case 2}: $t > m$ (so $w$ has more occurrences of $b$ than of $\binv$).  
Then $r = t-m > 0$.
Recall that we are denoting the fractional representations of $g$
and $g^n$ by $b^r \alpha^u$ and $b^{rn} \alpha^{u_n}$, respectively,
where $u_n = u(1 + k^r + \cdots + k^{(n-1)r})$. Since $m$ is the number of
components in the central subword of $\NFB(g)$, we have $k^m u \in \Z$.
So, by choosing $n' >0$ such that $(n'-1)r \ge m$, we have
$k^{(n-1)r}u \in \Z$ for all $n \ge n'$. Then the fractional
part of $u_n$ remains constant for all $n \ge n'$, and hence the central
subword of $\NFB(g^n)$ is the same word $w_\ct'$ for all such $n$,
and also the sequence $s_n = \lfloor u_n \rfloor$ is strictly increasing
for $n \ge n'$.

So for $n \ge n'$, if $\NFB(g^n) = b^{t_n} w_\ct' \aA^{s_n}$, then
$$b^{t_{n+1}} w_\ct' \aA^{s _{n+1}} =_\BS
 b^{t_n} w_\ct' \aA^{s_n} b^r\aA^{u}$$
and so $t_{n+1} = t_n+r$ and
$$\aA^{s _{n+1}} =_\BS w_\ct'^{-1}b^{-r}w_\ct' \aA^{s_n}b^r \aA^{u} =_\BS
[w_\ct',b^r] \aA^{k^rs_n+u}.$$

Now the commutator $[w_\ct',b^r]$ is some fixed power $\aA^\kappa$ of $\aA$
for some $\kappa \in \Q$, and $s_{n+1} = k^r s_n + \lambda$ where
$\lambda := u + \kappa$ is a constant that must lie in $\Z$.

It follows from $t_{n+1}=t_n+r$ that $t_n = rn+c$ for some constant $c \in \Z$.
Since we also have $s_{n+1} = k^r s_n + \lambda$, we can apply
Lemma~\ref{lem:expedt0l}
with $w = w_\ct'$ to deduce that $\{ \NFB(g^n) : n \ge n' \}$  and
hence also $\{ \NFB(g^n) : n \geq 0\}$ is \E.
\end{proof}

\begin{proposition}
\label{prop:conj}
The set of conjugators of fixed pairs of elements is \E;
that is, for fixed $g,h \in \BS$,
the set $\Con_{hg} := \{  w \in \NFB : w h =_\BS g w \}$ is \E.
\end{proposition}
\begin{proof}

The solution set is either empty or a right coset $C_g x$ of the centraliser
$C_g$ of $g$, for some fixed $x \in \BS$. Let $b^{r_x} \aA^{u_x}$
(where $u_x$ might be negative) be the fractional representation of $x$.

If $g=1$ then $\Con_{hg} = \NFB$ when $h=1$ and is empty otherwise and,
if $g \in N \setminus \{1\}$, then $C_g = N$, and $\Con_{hg}(hg)$ is either
empty, or equal to the set of normal form words for which $t-m = r_x$.
The result holds in these cases by Lemma~\ref{lem:regedt0l} and
Proposition~\ref{prop:fixedr}.

Otherwise, as we saw in the previous proof, $C_g = C_1 \cup C_2$ is the
disjoint union of two sets, where $C_1$ (from Case 1) is regular and $C_2$ is the union of
a finite set with a set $C_2'$ of the form
$\{ b^{rn+c} w_\ct' \aA^{s_n} : n \ge n' \}$ for some fixed $n' \ge 0$
where $r>0$ and, for $n \ge n'$, we have $s_{n+1} = k^r s_n + \lambda$
and $c,\lambda \in \Z$ are constants.

Suppose that $\Con_{hg}$ is nonempty. Then, by a similar argument to that used
in the proof of Case 1 of Proposition~\ref{prop:BScent}, we see that,
for all but finitely many of the words in $C_1$, multiplication on the right by
$x$ affects only a suffix of  bounded length, and so $\NFB(C_1x)$ is regular
and hence  \E by Lemma~\ref{lem:regedt0l}.

It remains to consider the set $\NFB(C_2'x)$ with $C_2'$ as above.
If $r_x \ge 0$ then, for $n \ge n'$ and $g =b^{rn+c}w_\ct'\aA^{s_n} \in C_2$,
we have $gx =_\BS b^{rn+c} w_\ct'b^{r_x}\aA^{u_x} \aA^{k^{r_x}s_n}$.
Now $\NFB(b^{rn+c} w_\ct'b^{r_x}\aA^{u_x}) = b^{rn+c'} w_\ct'' \aA^{\lambda_x}$
for some constant central subword $w_\ct''$ and $c,\lambda_x \in \Z$
and so, for sufficiently large $n$, we have
$\NFB(gx) = b^{rn+c'} w_\ct'' \aA^{s_n'}$, where
$s_n' = k^{r_x}s_n + \lambda_x$. Then $s_n'$ also satisfies the recurrence
relation $s_{n+1}' = k^r s_n' + \lambda'$ for some constant $\lambda' \in \Z$,
and so the set $\NFB(C_2'x)$ is \E by Lemma~\ref{lem:expedt0l}.

Now suppose that $r_x < 0$. Since, for $t \ge 0$, we have
$s_{n+t} = k^{rt}s_n + (k^{r(t-1)} + \cdots + k^r + 1)\lambda$, we see that
$s_{n+t} \bmod k^{-r_x}$ is constant for all sufficiently large $t$
(i.e. such that $rt \ge -r_x$). So, for sufficiently large $n$, we have
$\NFB(\aA^{s_n} b^{r_x}) = w_\ct'' \aA^{\lfloor  k^{r_x}s_n \rfloor}$ for
some fixed central subword $w_\ct''$, and hence, by splitting $u_x$ into its
integral and fractional parts, we see that
$\NFB(gx) = b^{rn+c'} w_\ct''' \aA^{s_n'}$ for fixed central subword $w_\ct'''$,
where $s_n' = \lfloor k^{r_x}s_n \rfloor + \lambda_x$ for constants
$c',\lambda_x \in \Z$.  Then since, as we saw above, $s_n \bmod k^{-r_x}$ is
constant for sufficiently large $n$, we have $s_{n+1}' = k^r s_n' + \lambda'$
for some constant $\lambda' \in \Z$, for sufficiently large $n$,
and the result follows from Lemma~\ref{lem:expedt0l}.
\end{proof}

\section{Multiplication and inversion}
\label{sec:BSmultInv}
In this section we shall prove the following two results.
Note that (informally) the second of these results says that the
multiplication table of the group with respect to $\NFB$ is \E.
This result was partly motivated by a related result of Gilman,
who proved in \cite{Gilman02} that a group is hyperbolic if and only if its
multiplication table is context-free with respect to some regular normal form.

\begin{theorem}\label{thm:BSmult}
The language $\{x\#y\#z : x,y,z \in \NFB,\, xy=_\BS z \}$ is \E.
\end{theorem}

\begin{theorem}\label{thm:BSmultinv}
The language $\{x\#y\#z : x,y,z \in \NFB,\, xy=_\BS z^{-1} \}$ is \E.
\end{theorem}

In the proofs, we first prove that the corresponding subsets of $\NFFB$ are \E,
which reduces essentially to addition and subtraction of numbers in base $k$.
In order to derive a corresponding proof for $\NFB$, we need to convert 
 a positive integer $n$ written in base $k$ to the string $\alpha^n$
 (with $\alpha = a$ or $\ainv$) and then simulate the above addition and subtraction on the
 exponents of these strings. It might be helpful to illustrate the 
conversion process, from $\NFFB$ to $\NFB$, in the special case $x=a$, $y=b^t \alpha^s$, and we shall
do that case first.

\begin{lemma}\label{lem:leftmulta}
The set
\[L:=\{ a\#u\#v: v= \NFB(au),\,u = b^t \aA^s,\,\aA \in \{a,\ainv\},
\,t,s \in \N_0 \}\]
is \E.
\end{lemma}
\begin{proof}
We do this by partitioning $L$ into three subsets, and prove that
each of these is \E. The result then follows from Lemma~\ref{lem:closure}.

The first of these subsets is
$\{ a\#u\#v \colon u = b^t a^s, s,t \ge 0, v=\NFB(au)\} =
\{ a \# b^t a^s \# b^t a^{s+k^t} : s,t \ge 0 \}$. To show that this set is \E,
we need extra symbols $\sa$ and $T$. We start with the start word 
$T \# \sa \# \sa a$ and apply $\sct_1^*\sct_2$, for tables $\sct_1 := 
\{ \sa \to b \sa, a \to a^k \}$ and $\sct_2 := \{ T \to a \}$. 
This produces the word $a\#b^t\sa\#b^t\sa a^{k^t}$, where $t \ge 0$ is the
number of applications of $\sct_1$.  
Then we apply $\sct_3^*\sct_4$, for tables $\sct_3 := \sa\rightarrow a\sa$
and  $\sct_4 :=\sa\rightarrow \emptystring$.

Note that $ab^t\ainv^s =_\BS b^t a^{k^t - s}$, so
$\NFB(ab^t\ainv^s) = b^t \ainv^{s - k^t}$ when $s \ge k^t$, and
$\NFB(ab^t\ainv^s) = b^t a^{k^t - s}$ when $k^t > s$.

Our second of the three subsets is
$\{ a\#b^t\ainv^s\#  b^t \ainv^{s-k^t} : s, t\ge 0, s \ge k^t \}$.
To show that this
is \E, observe that we can construct arbitrary words in this language by first
constructing words of the form $a\# b^t \ainv^{k^t}\sA \# b^t\sA$, using a
similar construction as in the preceding case.
Then we apply $\sct_3^*\sct_4$ for tables $\sct_3 := 
\{ \sA\rightarrow \ainv\sA \}$ and $\sct_4 := \{\sA\rightarrow \emptystring \}$.

Proving that our third subset,
$\{ a\#b^t\ainv^s\#  b^t a^{k^t-s} : s,t>0, k^t > s \}$ is \E is more difficult,
and this is the case that we are using to illustrate how we simulate
subtraction of base $k$ numbers. (In fact the previous two cases can also be
done using this technique, but they were easier to do directly.)

We shall describe a recipe for constructing all words of this form, which is
based on the idea of  carrying out the subtraction $k^t-s$ using the
representations of $k^t$ and $s$ in base $k$.

Let $s = s_0 + s_1k + \cdots + s_{t-1}k^{t-1}$ be the
expansion of $s$ in base $k$.  Then, for some $j$ with $0 \le j \le t-1$,
we have $s_i = 0$ for $0 \le i < j$, and $s_j \ne 0$
(we are assuming that $s>0$).
Then $k^t-s = s_0' + s_1'k + \cdots + s_{t-1}'k^{t-1}$, where:
\begin{mylist}
\item[(i)] $s_i' = 0$ for $0 \le i < j$;
\item[(ii)] $s_j' = k-s_j$;
\item[(iii)] $s_i' =  k-s_i-1$ for $j<i \le t-1$.
\end{mylist}

The construction of this word involves symbols $\sa$ and $\sA$
(variables in the associated \E system), which represent potential occurrences
of $a$ and $\ainv$, and will eventually be deleted.
The construction consists of $t$ steps, numbered $1,\ldots,t$.

As axiom we use the word $a \# \sA\#\sa$.

In Step $i$, for $1 \le i \le t$, we do the following:
\begin{mylist}
\item[(i)] apply the rule $\# \rightarrow \#b$;
\item[(ii)] apply the rule $\sA \rightarrow \sA^k \ainv^{s_{i-1}}$;
\item[(iii)] apply the rule $\sa \rightarrow \sa^k a^{s'_{i-1}}$.
\end{mylist}
Apply $\sa\rightarrow \emptystring$, and $\sA\rightarrow \emptystring$. 

Note that since $s_i,s'_i \in \{0,1,\ldots,k-1\}$ the rules we apply come from a finite set.

More formally, define tables as follows. 
\begin{align*}
  \alpha\colon & \#\mapsto\# b,~\sA\mapsto\sA^k,~\sa\mapsto\sa^k \\
  \delta\colon &\sa\mapsto\varepsilon,~\sA\mapsto\varepsilon
\end{align*}
and for each \(0\leq m\leq k-1\):
\begin{align*}
	\beta_m\colon & \#\mapsto\# b,~\sA\mapsto\sA^k\ainv^m,~\sa\mapsto\sa^k a^{k-m} \\
	\gamma_m\colon & \#\mapsto\# b,~\sA\mapsto\sA^k\ainv^m,~\sa\mapsto\sa^k a^{k-m-1} .\\
\end{align*}
Then the \E system with the following rational control produces the required language in the way described above:
\[\alpha^*(\beta_1\mid\beta_2\mid\cdots\mid \beta_{k-1})(\gamma_0,\gamma_1,\ldots,\gamma_{k-1})^*\delta.\]

\end{proof}

\begin{proofof}{Theorem~\ref{thm:BSmult}}
For words $x,y,z \in \NFB$, we denote the words $\NFFB(x)$, $\NFFB(y)$ and
$\NFFB(z)$ by $\fx$, $\fy$ and $\fz$, respectively.
We prove first that the language
$\LF = \{\fx\#\fy\#\fz : x,y,z \in \NFB,\, xy=_\BS z \}$ is \E,
and then explain how to adapt the arguments to prove the theorem.

Let $x,y,z \in \NFB$ with $xy=_\BS z$, and suppose that the fractional
representations of $x$ and $y$ are $b^{r_x} a^{u_x}$ and $b^{r_y} a^{u_y}$,
respectively (where $r_x, r_y, u_x$ and $u_y$ could be positive or negative).
Then $z$ has fractional representation $b^{r_x+r_y} a^{u_z}$ with
$u_z = k^{r_y}u_x+u_y$.
As usual, we aim to construct the language $\LF$ using an \E system, starting
with the word $\ud \# \ud \#\ud$.
Recall that $\fx$ consists of $b^{r_x}$ or $\binv^{-r_x}$ followed by the base $k$
representation of $u_x$ written backwards, and similarly for $\fy$ and $\fz$.

There are various cases to be considered, depending on the signs of
$r_x$, $r_y$, $u_x$ and $u_y$.
We need to partition $\LF$ into a large number of disjoint subsets depending
on these signs, and the \E systems that define these subsets are all slightly
different. If $u_x$ and $u_y$ have different signs, then the sign of
$u_z = k^{r_y}u_x+u_y$ may be positive or negative, and we need to distinguish
between those cases.
So there are $24$ principal cases. Each of $r_x$ and $r_y$ can be
non-negative or negative, and for each of these four possibilities there are
six subcases: $u_x,u_y \ge 0$;
$u_x,u_y \le 0$; $u_x>0,u_y<0,u_z\ge 0$; $u_x>0,u_y<0,u_z< 0$; 
$u_x<0,u_y>0,u_z\ge 0$; and $u_x<0,u_y>0,u_z< 0$.

The fractional and integral parts of $\fz$ are computed by addition or
subtraction (depending on the signs of $u_x$ and $u_y$) of those of
$\fx$ and $\fy$, where the radix point in that of $\fx$
is shifted $r_y$ places to the left (or $-r_y$ places to the right)
before performing this operation.
Each of the possible combinations of signs of $u_x$, $u_y$ and $u_z$
constitutes a separate subcase, and in the first step of the construction
we insert the signs of $u_x$, $u_y$ and $u_z$ for the subcase that we are
dealing with.

In the subsequent steps, we carry out the addition or subtraction of
the base $k$ numbers, dealing with one base $k$ digit in each step,
working from left to right (i.e. from the smallest power of $k$ to the largest).
Some of these operations will result in a ``carry one'' that needs to be
handled in the usual way in the following step. So these ``carry'' steps
must be followed by a further step that might itself be a carry step.
(In the formal \E system, we need extra variables to indicate that the
next step should be a ``carry'' step. We have left out the details of that
process in this proof, but we will present an explicit system for one case of
the language $L$ in Subsection~\ref{sec:BSexplicit} below.)

After completing this addition or subtraction, there may be some further steps
in which powers of $b$ or of $\binv$ are inserted at the left of
$\fx'$ and $\fz'$
(where  $\fx'\#\fy'\#\fz'$ denotes the word that has been constructed so far).

Rather than attempting formal proofs in all cases, we shall content ourselves
with providing the constructions of $\fx\#\fy\#\fz$ for three illustrative
examples.  The value of $k$ is not critical, and we take $k=3$ in our examples.

The easiest situation is when $r_y=0$, and we are just calculating
$u_x + u_y$ in base $k$ arithmetic. Suppose, for example, that
$\fx = \binv \up 11\ud 21$ and $\fy= \um 101 \ud 2$,
so $x = b a\binv a\binv a^5 =_\BS \binv a^{49/9}$, $y= b^3 \ainv \binv \binv \ainv \binv  \ainv ^2 =_\BS a^{-64/27}$.
Hence  $z =_\BS \binv a^{83/27}$, $\fz = \binv  \up 200 \ud 01$, and and $z = b^2 a^2\binv ^3a^3$.
Then the construction is

{\footnotesize
\[
\begin{array}{cccccc}
\ud \# \ud \#\ud&\to&
 \up\ud \# \um\ud  \# \up\ud &\to\\
 \up\ud \# \um1 \ud  \# \up2 \ud &\to&
 \up1 \ud \# \um10 \ud  \# \up20 \ud &\to\\
 \up11 \ud \# \um101 \ud  \# \up200 \ud  &\to&
 \up11 \ud 2\# \um101 \ud 2 \# \up200 \ud  0 &\to\\
 \up11 \ud 21\# \um101 \ud 2 \# \up200 \ud  01 &\to&
\binv  \up11 \ud 21\# \um101 \ud 2 \# \binv \up200 \ud 01
\end{array}
\]
}

In this example, there is just one carry step, namely
$ \up\ud \# \um\ud  \# \up\ud \to \up\ud \# \um1 \ud  \# \up2 \ud$.

Now let us keep the same $\fx$, but replace $\fy$ by $b \um101 \ud 2$, so
now $y = b^4 \ainv \binv\binv\ainv \binv \ainv^2 =_\BS ba^{-64/27}$,
$z = b^3 a^2\binv a^2\binv a^2\binv a^{13} =_\BS  a^{377/27}$ and $\fz = \up 222\ud 111$.
Now the construction is

{\footnotesize
\[
\begin{array}{cccccc}
\ud \# \ud \#\ud&\to&
\up  \ud \# \um \ud  \# \up\ud &\to\\
\up \ud \# \um1 \ud  \# \up2 \ud &\to&
\up \ud \# \um10 \ud  \# \up22 \ud  &\to\\
\up1 \ud \# \um101 \ud  \# \up222 \ud  &\to&
\binv \up11 \ud \#  b \um101 \ud 2 \# \up222 \ud 1 &\to\\
\binv \up11 \ud 2\# b \um101 \ud 2 \# \up222 \ud 11 &\to&
\binv \up11 \ud 21\# b \um101 \ud 2 \# \up222 \ud 111 
\end{array}
\]
}

Note that, when $r_y > 0$, we insert a symbol $b$ at the beginning of $\fy'$ in
each of the $r_y$ steps  in which we are processing fractional parts of $x$
and integral parts of $y$. If there are symbols $\binv$ to be entered at the
beginning of $\fx'$, then (as we did in one of the steps in the example above)
we insert a $\binv$ in $\fx'$ in the same step as the $b$ in $\fy'$;
otherwise we would insert a symbol $b$ at the beginning of $\fz'$.
In this example the first four subtraction steps are all carry steps.

As illustrated in the first and following example, any remaining occurrences
of $b$ or $\binv$ at the beginning of $\fx$ can be inserted into
$\fx'$ and $\fz'$ at the end of the construction.

Now we consider an example with $r_y < 0$, and with changed signs
for $x$, $y$ and $z$, namely
$\fx = \binv \um21\ud 112$ and $\fy=\binv \up221 \ud 01$,
so $x = b\ainv^2\binv\ainv \binv \ainv^{22} =_\BS \binv a^{-203/9}$,
$y =  b^2 a^2\binv a^2\binv a\binv a^3 = _\BS \binv a^{98/27}$,
$z = \ainv^2\binv\ainv^2\binv \ainv^{3} =_\BS \binv^2 a^{-35/9}$ and $\fz = \binv^2 \um22 \ud 01$.
The construction is

{\footnotesize
\[
\begin{array}{cccccc}
\ud \# \ud \#\ud&\to&
 \um \ud \# \up \ud  \# \um\ud &\to\\
 \um2 \ud \# \up2\ud  \#   \um\ud &\to&
 \um21 \ud \# \up22 \ud  \# \um2 \ud &\to\\
 \um21 \ud 1 \# \binv \up221 \ud  \# \binv \um22 \ud &\to&
 \um21 \ud 11\# \binv \up221 \ud 0 \# \binv \um22 \ud 0 &\to\\
 \um21 \ud 112\# \binv \up221 \ud 01 \# \binv \um22 \ud 01 &\to&
\binv \um11 \ud 112\# \binv \up221 \ud 01 \# \binv^2 \um22 \ud 01
\end{array}
\]
}

Here we inserted $\binv$ into $\fy'$ in the step dealing with the integral
part of $\fx$ and the fractional part of $\fy$. Since there is no $b$ at the
beginning of $\fx$, we insert $\binv$ into $\fz'$ at the same time.

Now we turn to the  proof that the language
$L = \{x\#y\#z : x,y,z \in \NFB,\, xy=_\BS z \}$ in the theorem statement
is \E.  Again we denote the subwords of $x,y,z$ that have been inserted into
$x\#y\#z$ so far by $x',y',z'$.

The words $x',y',z'$ are constructed in the same way and roughly in the
same order as the corresponding subwords $\fx',\fy',\fz'$ of $\fx,\fy,\fz$;
that is, for each step in the process of constructing $\fx',\fy'$ and $\fz'$,
there is a corresponding step in the construction of $x',y'$ and $z'$.

There are two principal issues that arise here. One of these involves
the insertion of the right subwords of the words $x, y, z \in \NFB$.
These subwords are strings in the generators $a$ or $\ainv$ of whichh
lengths are represented by numbers in base $k$ in $\fx,\fy,\fz$.
The integral part of $\fz$ is calculated from $\fx$ and $\fy$ by
addition or subtraction of numbers in base $k$.
The corresponding process for $x,y$ and $z$ can be carried out using the
method illustrated in the proof of Lemma~\ref{lem:leftmulta} that involves
the use of dummy symbols $\sa$ and $\sA$ that will be deleted at the end of
the process.

The second issue concerns the insertion of the central subwords, together with
parts of the left subwords, of $x,y,z$ into $x',y',z'$. In general,
in a step in which we insert digits $i$ with $0 \le i < k$ into the fractional
parts of each of $\fx'$, $\fy'$ and $\fz'$, we insert $a^i\binv$ into the
central subwords of $x,y$ and $z$ at their right hand ends, immediately before
the radix point. However, we want to ensure that the same total
powers of $b$ are inserted into $x'y'$ as into $z'$, which we do as follows.

In such a step, if there are letters $b$ in the left subwords of $x$ and/or $y$
that have not yet been inserted into $x'$ and/or $y'$, then we insert $b$ into
$x'$ and/or $y'$ in the same step.  If this involves inserting $b$ into both
$x'$ and $y'$ then, since we are also inserting terms $a^i\binv$ into
$x'$ and $y'$, the total power of $b$ inserted in $x'y'$ is zero and,
since we are inserting a term $a^i\binv$ into $z'$, we insert a $b$ in the
left subword of $z'$ to ensure that the total power of $b$ entered into $z'$
is also zero. But if we insert $b$ into just one of $x'$ and $y'$ then the total
power of $b$ inserted in $x'y'$  is $-1$, and we do not insert $b$ into $z'$.

In some steps, we may be processing digits in the fractional part of
one of $\fx$ and $\fy$ and in the integral part of the other, and again we
just need to ensure that we insert the same total power of $b$ into
$x'y'$ and into $z'$.  We would be in trouble if there was no $b$ to insert
either into $x'$ or into $y'$ and we had to insert a term $a^i\binv$ into
$z'$, but in fact that never happens.
That situation could only arise when $r_y < 0$, and in that case we would
be combining a fractional digit of $\fy'$ with an integral digit of $\fx'$,
so we would not be changing the fractional part of $x'$. Note also that if
the above process should involve inserting $b\binv$ at the beginning of the word
$z'$ then we would of course not do that (that situation arises in the third
example below).

As was the case with $\fx',\fz'$, any remaining occurrences of $b$ or $\binv$
in the left subwords of $x',z'$ can be inserted at the end of the process.

Let us now illustrate the procedure with the same three examples as before.
As in the proof of Lemma~\ref{lem:leftmulta}, we use dummy symbols
$\sa_x,\sa_y,\sa_z$ and $\sA_x,\sA_y,\sA_z$, which will eventually be removed,
to represent future possible instances of $a$ and of $\ainv$ in
$x, y, z$, respectively. (In fact we have shortened the process
by removing these dummy symbols in the final step of the rest of the procedure
rather than in a separate final step at the end.)
The first of the examples was
$x = b a\binv a\binv a^5$, $y= b^3 \ainv \binv\binv\ainv \binv \ainv^2$,
$z = b^2 a^2\binv^3a^3$.
The construction is

{\footnotesize
\[
\begin{array}{cccccc}
\# \#&\to& \# b\ainv \binv  \# ba^2\binv &\to\\
ba\binv \# b^2\ainv \binv\binv \# b^2a^2\binv^2&\to&
  ba\binv a\binv \# b^3\ainv \binv\binv\ainv \binv \# b^2a^2\binv^3&\to\\
ba\binv a\binv a^2 \sa_x^3 \# b^3\ainv \binv\binv\ainv \binv \ainv^2 \sA_y^3  \# b^2a^2\binv^3\sa_z^3 &\to&
  ba\binv a\binv a^5 \# b^3\ainv \binv\binv\ainv \binv \ainv^2  \# b^2a^2\binv^3 a^3 &
\end{array}
\]
}

The second example is $x = b a\binv a\binv a^5$, $y = b^4 \ainv \binv\binv\ainv \binv \ainv^2$,
$z = b^3 a^2\binv a^2\binv a^2\binv a^{13}$.

{\scriptsize
\[
\begin{array}{cccccc}
\# \#&\to& \# b\ainv \binv  \# ba^2\binv &\to\\
\# b^2\ainv \binv\binv  \# b^2a^2\binv a^2\binv &\to& 
  ba\binv \# b^3\ainv \binv\binv\ainv \binv  \# b^3 a^2\binv a^2\binv a^2\binv &\to\\ 
ba\binv a\binv \# b^4\ainv \binv\binv\ainv \binv \ainv^2\sA_y^3  \# b^3a^2\binv a^2\binv a^2\binv a \sa_z^3 &\to& 
  ba\binv a\binv a^2 \sa_x^3 \# b^4\ainv \binv\binv\ainv \binv \ainv^2\sA_y^9  \# b^3a^2\binv a^2\binv a^2\binv a^4\sa_z^9&\to \\
ba\binv a\binv a^5 \# b^4\ainv \binv\binv\ainv \binv \ainv^2   \# b^3a^2\binv a^2\binv a^2\binv a^{13} &
\end{array}
\]
}

The third example is $x = b\ainv^2\binv\ainv \binv \ainv^{22}$, $y =b^2 a^2\binv a^2\binv a\binv a^3$,
$z = \ainv^2\binv\ainv^2\binv \ainv^{3}$.

{\scriptsize
\[
\begin{array}{cccccc}
\# \#&\to&   b\ainv^2\binv\# ba^2\binv  \#  &\to\\
b\ainv^2\binv\ainv \binv \# b^2a^2\binv a^2\binv  \# \ainv^2\binv  &\to& 
  b\ainv^2\binv\ainv \binv \ainv\sA_x^3 \# b^2a^2\binv a^2\binv a\binv  \# \ainv^2\binv\ainv^2\binv  &\to\\ 
b\ainv^2\binv\ainv \binv \ainv^4 \sA_x^9 \# b^2a^2\binv a^2\binv a\binv \sa_y^3  \# \ainv^2\binv\ainv^2\binv  \sA_z^3  &\to& 
b\ainv^2\binv\ainv \binv \ainv^{22}  \# b^2a^2\binv a^2\binv a\binv a^3  \# \ainv^2\binv\ainv^2\binv \ainv^3  &
\end{array}
\]
}

An explicit \E system for one case of the proof is presented in
Subsection~\ref{sec:BSexplicit}
\end{proofof}

\begin{proofof}{Theorem~\ref{thm:BSmultinv}}
Let $x,y,z \in \NFB$ with $xy=_\BS z^{-1}$ and let $\bar{z} = \NFB(xy)$.
Then, as we saw in the previous theorem, the fractional representations of
$x$ and $y$ can be written as  $x=_\BS b^{r_x}a^{u_x}$, $y=_\BS b^{r_y}a^{u_y}$,
and then $\bar{z} =_\BS b^{r_x+r_y} a^{u_{\bar{z}}}$ with $u_{\bar{z}} = k^{r_y}u_x+u_y$.

Now we have $z =_\BS \bar{z}^{-1} =_\BS a^{-u_{\bar{z}}} b^{-(r_x+r_y)} = 
b^{-(r_x+r_y)} a^{u_z}$ with $u_z = - k^{-(r_x+r_y)}u_{\bar{z}}$.
So $u_{z}$ has the opposite sign to that of $u_{\bar{z}}$, and the digits of
its (reversed) base $k$ expansion are the same as those of $\bar{z}$, but the
radix point is shifted $r_x+r_y$ places to the right or $-(r_x+r_y)$ places to
the left.

Let $L :=  \{x\#y\#z : x,y,z \in \NFB,\, xy=_\BS z^{-1} \}$ 
and $\LF :=  \{\fx\#\fy\#\fz : \fx,\fy,\fz \in \NFFB,\, xy=_\BS z^{-1} \}$ 
The processes involved in proving first that $\LF$ and then that $L$ are \E
are similar to those in the previous proof, with the added complication that
power $b^{r_x}$ of $b$ in $\fx$ plays a more significant role. 

The constructions of $x\#y\#z$ and of $\fx\#\fy\#\fz$ using \E systems split
up into four phases, some of which may be empty in some examples. In the first
phase we are constructing parts of the fractional parts of $\fx$, $\fy$ and $\fz$,
and in the fourth phase we are constructing parts of their integral parts.
In the second and third phases, we are constructing  parts of the
fractional parts of some of them and of the integral parts of the others.
As in the previous proof, we denote the words constructed so far by
$x'\#y'\#z'$ and $\fx'\#\fy'\#\fz'$.

In the construction of $x,y,z$, in each step in the first phase we
insert a $b$ at the beginning of each of $x',y',z'$, and $\alpha^i\binv$ for some
$0 \le i < k$ at the end of its fractional part, but with the usual proviso
that we avoid inserting $b\binv$ at the beginning of the word.

Unlike in the proof of Theorem~\ref{thm:BSmult}, in the situation when
$r_x$ and $r_y$ have opposite, there are two cases depending on the sign of
$r_x+r_y$. So there are six possibilities for $r_x$ and $r_y$ to be
considered, each of which splits into six subcases for $u_x$ and $u_y$
as in Theorem~\ref{thm:BSmult}. So there are $36$ cases in total.
We shall describe phases 2 and 3 in more detail in three of these
possibilities for $r_x,r_y$, and provide
the constructions for an illustrative example with $k=3$ in each case.

Suppose first that $r_x$ and $r_y$ are both non-negative.
In the second phase, we are constructing parts of the fractional parts of
$\fx$ and $\fz$, and of the integral part of $\fy$. This phase consists of $r_y$
steps, in each of which we insert a $b$ at the beginning of each of $x'$,
$y'$ and $\fy'$, a $\binv$ at the beginning of $\fz'$, and one term $a^i\binv$ at the
end of the fractional parts of each of $x'$ and $z'$.

In the third phase, we are constructing parts of the integral parts of $\fx$ and
$\fy$, and of the fractional part of $\fz$. This phase consists of $r_x$ steps,
in each of which we insert a $b$ at the beginning of $x'$ and of $\fx'$, a $\binv$
at the beginning of $\fz'$, and a term $a^i\binv$ at the end of the fractional part
of $z'$.

As an example, we take $x=b^3\ainv \binv\ainv^2$, $y=b^2a^2\binv a$, so $\fx = b^2 \um1\ud 2 $,
$\fy = b\up2\ud 1$, $\fz =\binv^3 \up 1210 \ud$, $z = ba\binv a^2\binv a\binv^2$.

Here are the derivations of $\fx\#\fy\#\fz$ and of $x\#y\#z$. There is one step in
each of the first two phases, two steps in the third phase, and none in the
fourth phase. (So the integral part of $\fz$ is $0$.)
We have saved space by suppressing the first step in which signs are
entered into $\fx',\fy'$ and $\fz'$.

{\small
\[
\begin{array}{cccccc}
\um\ud \# \up\ud \# \up\ud&\to&
 \um \ud \# \up2 \ud \# \up1 \ud&\to\\
\um1 \ud \# b \up2 \ud 1 \# \binv \up12 \ud &\to&
  b\um1 \ud 2\# b \up2 \ud 1 \# \binv^2 \up121 \ud &\to\\
b^2\um1 \ud 2\# b^2 \up2 \ud 1 \# \binv^3 \up1210 \ud &
\end{array}
\]
}

{\small
\[
\begin{array}{cccccc}
\#\# & \to&
  \# ba^2\binv \# ba\binv & \to\\
b\ainv \binv \# b^2 a^2\binv a \# ba\binv a^2\binv & \to&
  b^2\ainv \binv \ainv^2 \# b^2a^2\binv a \# ba\binv a^2\binv a\binv & \to\\
b^3\ainv \binv \ainv^2 \# b^2a^2\binv a \# ba\binv a^2\binv a\binv^2 &
\end{array}
\]
}

Suppose next that $r_x > 0$ and $r_y < 0$ and that $r_x > |r_y|$.
In the second phase, we are constructing parts of the fractional parts of
$\fy$ and $\fz$, and the integral part of $\fx$.  This phase consists of $|r_y|$
steps, in each of which we insert a $b$ at the beginning of each of $x'$, $\fx'$
and $z'$, a $\binv$ at the beginning of $\fy'$, and one term $a^i\binv$ into the end of
the fractional parts of each of $y'$ and $z'$.

In the third phase, we are constructing parts of the integral parts of $\fx$ and
$\fy$, and the fractional part of $\fz$. This phase consists of $r_x-|r_y|$ steps,
in each of which we insert a $b$ at the beginning of $\fx'$ and $x'$,
a $\binv$ at the beginning of $\fz'$, and a term $a^i\binv$ at the end of the
fractional part of $z'$.

As example, we take $x=b^3\ainv \binv\ainv^5$, $y=b\ainv \binv\ainv^2\binv\ainv^8$, so
$\fx = b^2 \um1\ud 21$, $\fy = \binv \um12\ud 22$,
$\fz =\binv \up 211 \ud 01$, $z = b^2 a^2\binv a\binv a\binv a^3$.

In the derivations below, there is one step in each of
the first three phases, and two in the fourth phase.

{\small
\[
\begin{array}{cccccc}
  \um\ud \# \um\ud \# \up\ud &\to&
  \um1 \ud \# \um1 \ud \# \up2 \ud &\to\\
b \um1 \ud 2\# \binv \um12 \ud \# \up21 \ud &\to&
b^2 \um1 \ud 21\# \binv \um12 \ud 2 \#  \binv \up211 \ud &\to\\
b^2 \um1 \ud 21\# \binv \um12 \ud 22 \#  \binv \up211 \ud 0 &\to&
b^2 \um1 \ud 21\# \binv \um12 \ud 22 \#  \binv \up211 \ud 01
\end{array}
\]
}

{\footnotesize
\[
\begin{array}{cccccc}
\#\# & \to &
  b\ainv \binv \# b\ainv \binv \# ba^2\binv & \to  \\
b^2 \ainv \binv \ainv^2 \sA_x^3 \# b\ainv \binv\ainv^2\binv \# b^2 a^2\binv a\binv &\to &
  b^3 \ainv \binv \ainv^5 \# b\ainv \binv\ainv^2\binv \ainv^2 \sA_y^3  \#  b^2 a^2\binv a\binv a\binv &\to \\
b^3 \ainv \binv \ainv^5 \# b\ainv \binv\ainv^2\binv \ainv^8  \#  b^2 a^2\binv a\binv a\binv \sa_z^3 &\to&
 b^3 \ainv \binv \ainv^5 \# b\ainv \binv\ainv^2\binv \ainv^8  \#  b^2 a^2\binv a\binv a\binv a^3
\end{array}
\]
}

The final case that we shall consider in detail is $r_x > 0$ and $r_y < 0$
with $r_x < |r_y|$.
As in the previous case that we considered, in the second phase, we are
constructing parts of the fractional parts of $\fy$ and $\fz$, and the integral
part of $\fx$.  But now this phase consists of $r_x$ steps, in each of which
we insert a $b$ at the beginning of each of $x'$, $\fx'$ and $z'$, a $\binv$ at the
beginning of $\fy'$, and one term $a^i\binv$ at the end of the fractional parts of
each of $y'$ and $z'$.

In the third phase, we are constructing parts of the integral parts of
$\fx$ and $\fz$, and the fractional part of $\fy$. This phase consists of
$|r_y|-r_x$ steps, in each of which we insert a $b$ at the beginning of
of $z'$ and $\fz'$, a $\binv$ at the beginning of $\fy'$, and a term $a^i\binv$ at the
end of the fractional part of $y'$.

As example, we take $x=b^3\ainv \binv\ainv^5$, $y=\binv a^2\binv a\binv a^5$, so
$\fx = b^2 \um1\ud 21$, $\fy = \binv^3 \up21\ud 21$,
$\fz =b \um 200 \ud 121$, $z = b^4 \ainv^2\binv\binv\binv \ainv^{16}$.

In the derivations below, there is one step in the first and third phases,
and two in the second and fourth phases.

{\footnotesize
\[
\begin{array}{cccccc}
\um\ud \# \up\ud \# \um\ud &\to&
  \um1 \ud \# \up\ud \# \um2 \ud &\to\\
b \um1 \ud 2 \# \binv \up\ud \# \um20 \ud &\to&
  b^2 \um1 \ud 21 \# \binv^2 \up2 \ud  \# \um200 \ud &\to\\
b^2 \um1 \ud 21 \# \binv^3 \up21 \ud  \# b \um200 \ud 1 &\to&
  b^2 \um1 \ud 21 \# \binv^3 \up21 \ud 2  \# b \um200 \ud 12 &\to\\
b^2 \um1 \ud  21 \# \binv^3 \up21 \ud  21  \# b \um200 \ud  121
\end{array}
\]
}

{\footnotesize
\[
\begin{array}{cccccc}
\#\# & \to &
  b\ainv \binv \# \# b\ainv^2\binv & \to \\
b^2 \ainv \binv \ainv^2 \sA_x^3 \# \binv \# b^2 \ainv^2\binv^2 & \to &
  b^3 \ainv \binv \ainv^5 \# \binv a^2\binv \# b^3 \ainv^2\binv^3 & \to \\
b^3 \ainv \binv \ainv^5 \# \binv a^2\binv a\binv \# b^4 \ainv^2\binv^3 \ainv \sA_z^3 & \to &
  b^3 \ainv \binv \ainv^5 \# \binv a^2\binv a\binv a^2\sa_y^3 \# b^4 \ainv^2\binv^3 \ainv^7 \sA_z^9 & \to \\
b^3 \ainv \binv \ainv^5 \# \binv a^2\binv a\binv a^5 \# b^4 \ainv^2\binv\binv\binv \ainv^{16}
\end{array}
\]
}

\end{proofof}

\section{Explicit system}\label{sec:BSexplicit}
We explicitly construct one of the \E systems described in the proof of Theorem \ref{thm:BSmult}, for the case \(\BS(1,3)\). The system described in this section generates the language
\[\{x\#y\#z\colon x,y,z\in\NFB,~u_x<0,u_y>0,u_z<0\}.\]
Systems that generate the corresponding languages for different signs of \(u_x, u_y, u_z\) are analogous, with the only difference occurring in the tables \(\alpha_{ij}, \beta_{ij}, \gamma_{ij}, \delta_{ij}\).

\begin{itemize}
	\item Alphabet: $\tilde{X},\tilde{Y},\tilde{Z},X,Y,Z,Z_1, X_b,Y_b,Z_b,\sA_X,\sa_Y,\sA_Z,\sA_{Z1},a,\ainv,b,\binv,\#$
	\item Terminals: $a,\ainv,b,\binv,\#$
	\item Axiom: $\tilde{X}\#\tilde{Y}\#\tilde{Z}$
	\item Tables: $\alpha_{ij}, \beta_{ij}, \gamma_{ij}, \delta_{ij}$ (for $i,j\in\{0,1,2\}$, defined in Appendix \ref{ap:tables}), 
	\begin{align*}
		\sigma\colon & \tilde{X}\mapsto \binv X, X\mapsto\binv X, \tilde{Z}\mapsto\binv Z, Z\mapsto\binv Z \\
		\rho_{X}\colon & X_b\mapsto bX_b \\
		\rho_{Y}\colon & Y_b\mapsto bY_b \\
		\rho_{XY}\colon & X_b\mapsto bX_b, Y_b\mapsto bY_b, Z_b\mapsto bZ_b\\
		\mu\colon & X_b\mapsto bX_b, Z_b\mapsto bZ_b \\
		\tau \colon & \tilde{X},X,X_b, \tilde{Y},Y,Y_b, \tilde{Z},Z,Z_b, \sA_X,\sa_Y,\sA_Z \mapsto \emptystring\\
	\end{align*}
\item Rational control: as in Figure~\ref{fig:RatControl}
\end{itemize}
Figure~\ref{fig:RatControl} is a schematic diagram of the finite state automaton defining the rational control. Labelled edges are single transitions in the FSA as usual. Unlabelled edges represent multiple transitions (each starting and ending at the same states as the unlabelled edge) as follows, where \(x\) is replaced with \(\alpha,\beta,\gamma,\delta\) as indicated in the corresponding dashed box. The arrow with the open triangle represents a transition from a non-carry state to a non-carry state, the open square represents non-carry to carry, the filled triangle carry to carry, and the filled square carry to non-carry.
\begin{center}
\begin{tabular}{cl}
	\includegraphics[width=35mm]{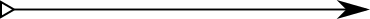} & 6 transitions, labelled \(x_{00}\mid x_{10}\mid x_{11}\mid x_{20}\mid x_{21}\mid x_{22}\)\\
	\includegraphics[width=35mm]{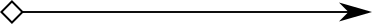} & 3 transitions, labelled \(x_{01}\mid x_{02}\mid x_{12}\)\\
	\includegraphics[width=35mm]{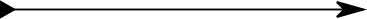} & 6 transitions, labelled \(x_{00}\mid x_{01}\mid x_{02}\mid x_{11}\mid x_{12}\mid x_{22}\)\\
	\includegraphics[width=35mm]{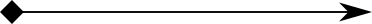} & 3 transitions, labelled \(x_{10}\mid x_{20}\mid x_{21}\) 
\end{tabular}
\end{center}
The edges connecting dashed boxes represent pairs of \(\varepsilon\)-labelled edges, connecting the shaded states within the boxes. The left hand shaded state in one box is connected to the left hand shaded state in the other box, and similarly for the right hand states. The right hand shaded states (marked with a dot) represent the fact that there is a carry that has yet to be resolved.

\begin{figure}
	\centering{
		\resizebox{128mm}{!}{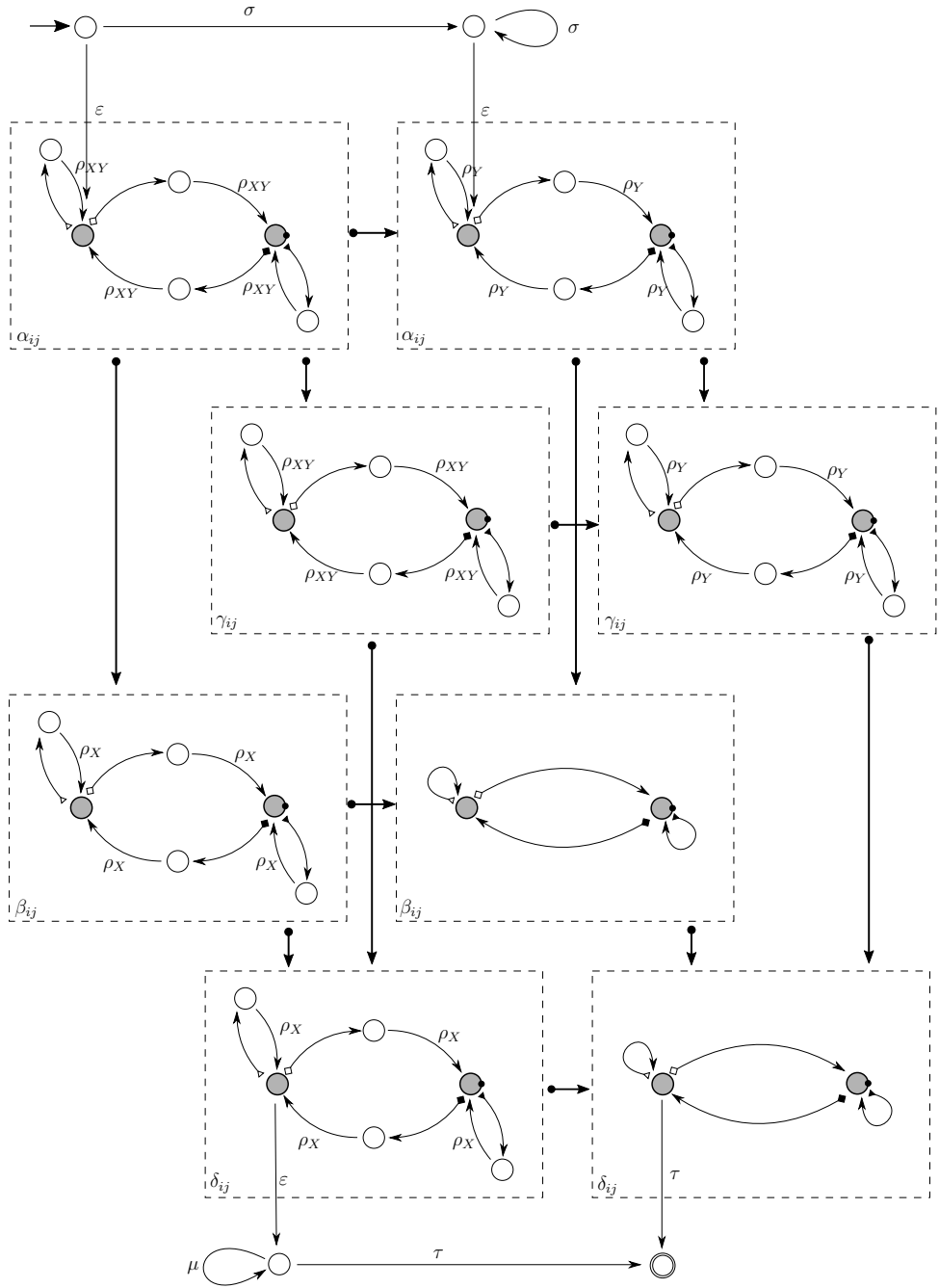}
	\caption{Rational control for the \E system of Theorem~\ref{thm:BSmult}}
	\label{fig:RatControl}}
\end{figure}
The variables \(X,Y,Z,Z_1\) are the `sources' of the powers of \(a\) and \(\ainv\) before the radix point, with \(Z_1\) recording a carry. The variables \(\sA_X, \sa_Y, \sA_Z, \sA_{Z1}\) are the `sources' of powers of \(a\) and \(\ainv\) after the radix point, with \(\sA_{Z1}\) again indicating a carry. Note that the rational control ensures that a word can never be completed with an outstanding carry.

The table \(\sigma\) optionally inserts \(\binv\)'s at the start of the central subwords of both \(x'\) and \(z'\). The table \(\mu\) optionally inserts \(b\)'s at the start of the same subwords. Note that the rational control ensures that these tables are never both used for the same word. The table \(\tau\) occurs at the end of every word in the rational control, with the purpose of deleting all remaining non-terminals.

The tables $\alpha_{ij}$ etc insert powers $\ainv^i$ and $a^j$ at the chosen step, as well as the \(\binv\) that separates them. For example, \(\alpha_{12}\) inserts \(\ainv \binv\) to the word representing \(x\), \(a^2\binv\) to the word representing \(y\), and either \(\ainv^2\binv\) with a carry, or \(\binv\) with a carry, to the word representing \(z\), depending on whether or not there was already a carry present. The choice of $\alpha,\beta,\gamma,\delta$ corresponds to both left and right words being before the radix point, only the right hand word being before the radix point, only the left, and neither, respectively. To ensure that the string \(b\binv\) never gets inserted, the variables \(\tilde{X}, \tilde{Y}, \tilde{Z}\) are used initially. Once a non-zero power of \(a\) or \(\ainv\) has been inserted, the corresponding variable is `initialised' and the tilde version is replaced.

The tables \(\rho\) add powers of \(b\) to the left hand side of each word, using the variables \(X_b,Y_b,Z_b\) as `sources'. 

Consider the third example in the proof of Theorem \ref{thm:BSmult} above: \(x=b\ainv^2\binv\ainv \binv\ainv^{22}\), \(y=b^2a^2\binv a^2\binv a\binv a^3\), \(z=\ainv^2\binv\ainv^2\binv\ainv^3\). This element of the language is produced by the \E system, via the word \(\alpha_{22}\rho_{XY}\alpha_{12}\rho_Y\beta_{11}\delta_{10}\delta_{21}\tau\) of the rational control, in the following way:
{\footnotesize
	\[
	\begin{array}{cccccc}
		\tilde{X}\#\tilde{Y}\#\tilde{Z} & \xtribarrow{\alpha_{22}} & X_b\ainv^2\binv X \# Y_ba^2\binv Y \# \tilde{Z} \\ 
		& \xrightarrow{\rho_{XY}} & 
		bX_b\ainv^2\binv X \# bY_ba^2\binv Y \# \tilde{Z} \\
		& \xdiarrow{\alpha_{12}} & bX_b\ainv^2\binv\ainv \binv X \# bY_ba^2\binv a^2\binv Y \# Z_b\ainv^2\binv Z_1 \\
		& \xrightarrow{\rho_Y} & bX_b\ainv^2\binv\ainv \binv X \# b^2Y_ba^2\binv a^2\binv Y \# Z_b\ainv^2\binv Z_1 \\
		& \xtriarrow{\beta_{11}} & bX_b\ainv^2\binv\ainv \binv\ainv\sA_X^3 \# b^2Y_ba^2\binv a^2\binv a\binv Y \# Z_b\ainv^2\binv\ainv^2\binv Z_1 \\
		& \xdibarrow{\delta_{10}} & bX_b\ainv^2\binv\ainv \binv\ainv(\ainv\sA_X^3)^3 \# b^2Y_ba^2\binv a^2\binv\ainv \binv\sa_Y^3 \# Z_b\ainv^2\binv\ainv^2\binv\sA_Z^3 \\
		& \xdiarrow{\delta_{21}}  & bX_b\ainv^2\binv\ainv \binv\ainv(\ainv(\ainv^2\sA_X^3)^3)^3 \# b^2Y_ba^2\binv a^2\binv\ainv \binv(a\sa_Y^3)^3 \# Z_b\ainv^2\binv\ainv^2\binv(\ainv\sA_Z^3)^3 \\
		& \xrightarrow{~\tau~} & b\ainv^2\binv\ainv \binv\ainv^{22} \# b^2a^2\binv a^2\binv\ainv \binv\ainv^3 \# \ainv^2\binv\ainv^2\binv\ainv^3
	\end{array}
	\]
}

\section{An equation in which the solution set might not be \E.}
\label{sec:BSnonEg}
We conjecture that the set $C = \{ x\#y : x,y \in \NFB,\,xy =_\BS yx \}$ is not \E.

Elements of $\BS(1,k)$ have the form $g=b^ra^u$, where the exponent $u$ of the
fractional part of $g$ lies in the set $\Z[1/k] = \{ i/k^m : i,m \in \Z\}$.
As we saw earlier, we have $g^n = b^{rn}a^{u_n}$ with
$u_n = u(1+k^{r} + \cdots + k^{(n-1)r}) = u\left(\frac{k^{rn}-1}{k^r-1}\right)$ for
$n \ge 0$.  Since $\frac{k^r-1}{k^{rn}-1} \not\in \Z[1/k]$ for $n>1$,
it follows that the element $b^ra$ is not a proper power for any $r \ge 0$.

So the set
\[ \{ b^ra\#\NFB( (b^ra)^n)  : r,n \ge 0 \} =
 \{ b^ra\# b^{rn} a^{\frac{k^{rn}-1}{k^r-1}}  : r,n \ge 0 \}
\]
is the intersection of $C$ with the regular set $b^*a\#b^*a^*$,
and it would suffice to show that this is not \E.

We conjecture that \cite[Theorem A]{Gilman96}  can be applied
to deduce that this set is not indexed; this would of course imply that it is not \E or even \NE.
We have verified by computer that, if the language were indexed, and we
applied that result with $m=2$, then we would deduce that the constant
$k$ in \cite[Theorem A]{Gilman96} satisfies $k \ge 300$.

\section*{Acknowledgements}
The authors wish to thank Graham Campbell for much useful discussion during the early stages of this project. Thanks are also due to the Heilbronn Institute for Mathematical Research for facilitating collaboration via their small grants scheme.

\appendix
\section{Table definitions}\label{ap:tables}
The following are the definitions of the tables \(\alpha_{ij}, \beta_{ij}, \gamma_{ij}, \delta_{ij}\) used in Section \ref{sec:BSexplicit}.
\begin{align*}
	\alpha_{00}\colon& X\mapsto \binv X  & \alpha_{01}\colon & X\mapsto \binv X & \alpha_{02}\colon & X\mapsto \binv X\\
	& Y\mapsto \binv Y && Y\mapsto a\binv Y && Y\mapsto a^2\binv Y \\
	& Z\mapsto \binv Z && \tilde{Y}\mapsto Y_ba\binv Y && \tilde{Y}\mapsto Y_b a^2\binv Y\\
	& Z_1\mapsto \ainv^2\binv Z_1 && Z\mapsto \ainv^2\binv Z_1 && Z\mapsto \ainv \binv Z_1\\
	& && \tilde{Z}\mapsto Z_b\ainv^2 \binv Z_1 && \tilde{Z}\mapsto Z_b\ainv \binv Z_1 \\
	& && Z_1\mapsto \ainv \binv Z_1 && Z_1\mapsto \binv Z_1\\ \\
	\alpha_{10}\colon & X\mapsto \ainv \binv X & \alpha_{11}\colon & X\mapsto \ainv \binv X & \alpha_{12}\colon & X\mapsto \ainv \binv X\\
	& \tilde{X}\mapsto X_b\ainv \binv X && \tilde{X}\mapsto X_b\ainv \binv X && \tilde{X}\mapsto X_b\ainv \binv X \\
	& Y\mapsto \binv Y && Y\mapsto a\binv Y && Y\mapsto a^2\binv Y\\
	& Z\mapsto \ainv \binv Z && \tilde{Y}\mapsto Y_ba\binv Y && \tilde{Y}\mapsto Y_ba^2\binv Y \\
	& \tilde{Z}\mapsto Z_b\ainv \binv Z && Z\mapsto \binv Z && Z\mapsto \ainv^2\binv Z_1\\
	& Z_1\mapsto \binv Z && Z_1\mapsto \ainv^2\binv Z_1 && \tilde{Z}\mapsto Z_b\ainv^2\binv Z_1 \\
	& &&  && Z_1\mapsto \binv Z_1\\ \\
	\alpha_{20}\colon & X\mapsto \ainv^2\binv X & \alpha_{21}\colon & X\mapsto \ainv^2\binv X & \alpha_{22}\colon & X\mapsto \ainv^2\binv X \\
	& \tilde{X}\mapsto X_b \ainv^2\binv X && \tilde{X}\mapsto X_b \ainv^2\binv X && \tilde{X}\mapsto X_b \ainv^2\binv X\\
	& Y\mapsto \binv Y && Y\mapsto a\binv Y && Y\mapsto a^2\binv Y \\
	& Z\mapsto \ainv^2\binv Z && \tilde{Y}\mapsto Y_ba\binv Y && \tilde{Y}\mapsto Y_b a^2\binv Y \\
	& \tilde{Z}\mapsto Z_b \ainv^2\binv Z && Z\mapsto \ainv \binv Z && Z \mapsto \binv Z\\
	& Z_1\mapsto \ainv \binv Z && \tilde{Z}\mapsto Z_b\ainv \binv Z && Z_1\mapsto \ainv^2\binv Z_1\\
	& && Z_1\mapsto \binv Z && \\
\end{align*}
\begin{align*}
	\beta_{00}\colon& \tilde{X},X,\sA_X\mapsto \sA_X^3 & 	\beta_{01}\colon& \tilde{X},X,\sA_X\mapsto \sA_X^3 & 	\beta_{02}\colon& \tilde{X},X,\sA_X\mapsto \sA_X^3 \\
	& Y\mapsto \binv Y & & Y\mapsto a\binv Y & & Y\mapsto a^2\binv Y \\
	& Z\mapsto \binv Z & & \tilde{Y}\mapsto Y_ba\binv Y & & \tilde{Y}\mapsto Y_ba^2\binv Y \\
	& Z_1\mapsto \ainv^2Z_1 & & Z\mapsto \ainv^2\binv Z_1 & & Z\mapsto \ainv \binv Z_1 \\
	& & & \tilde{Z}\mapsto Z_b\ainv^2\binv Z_1 & & \tilde{Z}\mapsto Z_b\ainv \binv Z_1 \\
	& & & Z_1\mapsto \ainv \binv Z_1 & & Z_1\mapsto \binv Z_1 \\ \\
	\beta_{10}\colon & \tilde{X},X,\sA_X\mapsto \ainv\sA_X^3 & \beta_{11}\colon & \tilde{X},X,\sA_X\mapsto \ainv\sA_X^3& \beta_{12}\colon & \tilde{X},X,\sA_X\mapsto \ainv\sA_X^3\\
	& Y\mapsto \binv Y & & Y\mapsto a\binv Y &  & Y\mapsto a^2\binv Y \\
	& Z\mapsto \ainv \binv Z & & \tilde{Y}\mapsto Y_ba\binv Y  & &  \tilde{Y}\mapsto Y_ba^2\binv Y\\
	& \tilde{Z}\mapsto Z_b\ainv \binv Z & & Z\mapsto \binv Z & & Z\mapsto \ainv^2\binv Z_1  \\
	& Z_1\mapsto \binv Z &&  Z_1\mapsto \ainv^2\binv Z_1 & & Z_1\mapsto \ainv \binv Z_1  \\ \\
	\beta_{20}\colon & \tilde{X},X,\sA_X\mapsto \ainv^2\sA_X^3 &  \beta_{21}\colon & \tilde{X},X,\sA_X\mapsto \ainv^2\sA_X^3 & \beta_{22}\colon & \tilde{X},X,\sA_X\mapsto \ainv^2\sA_X^3\\
	& Y\mapsto \binv Y & &  Y\mapsto a\binv Y & & Y\mapsto a^2\binv Y \\
	& Z\mapsto \ainv^2\binv Z & & \tilde{Y}\mapsto Y_ba\binv Y  & &  \tilde{Y}\mapsto Y_ba^2\binv Y  \\
	& \tilde{Z}\mapsto Z_b\ainv^2\binv Z & &  Z\mapsto \ainv \binv Z & &  Z\mapsto \binv Z  \\
	& Z_1\mapsto \ainv \binv Z & &  \tilde{Z}\mapsto Z_b\ainv \binv Z  & &  Z_1\mapsto \ainv^2\binv Z_1  \\
	& & & Z_1\mapsto \binv Z & &
\end{align*}
\begin{align*}
	\gamma_{00}\colon & X\mapsto \binv X & \gamma_{01}\colon & X\mapsto \binv X & \gamma_{02}\colon & X\mapsto \binv X \\
	& \tilde{Y}, Y, \sa_Y\mapsto\sa_Y^3& & \tilde{Y}, Y, \sa_Y\mapsto a\sa_Y^3 && \tilde{Y}, Y, \sa_Y\mapsto a^2\sa_Y^3\\
	& Z,\tilde{Z},\sA_Z\mapsto \sA_Z^3 & & Z,\tilde{Z},\sA_Z\mapsto \ainv^2\sA_{Z1}^3&& Z,\tilde{Z},\sA_Z\mapsto \ainv\sA_{Z1}^3\\
	& Z_1,\sA_{Z1}\mapsto \ainv^2\sA_{Z1}^3& & Z_1,\sA_{Z1}\mapsto \ainv\sA_{Z1}^3&& Z_1,\sA_{Z1}\mapsto \sA_{Z1}^3 \\ \\
	\gamma_{10}\colon & X\mapsto \ainv \binv X & \gamma_{11}\colon &  X\mapsto \ainv \binv X & \gamma_{12}\colon &  X\mapsto \ainv \binv X \\
	& \tilde{X}\mapsto X_b\ainv \binv X & & \tilde{X}\mapsto X_b\ainv \binv X  && \tilde{X}\mapsto X_b\ainv \binv X \\
	& \tilde{Y}, Y, \sa_Y\mapsto\sa_Y^3& & \tilde{Y}, Y, \sa_Y\mapsto a\sa_Y^3 && \tilde{Y}, Y, \sa_Y\mapsto a^2\sa_Y^3\\
	& Z,\tilde{Z},\sA_Z \mapsto \ainv\sA_Z^3 & & Z,\tilde{Z},\sA_Z\mapsto \sA_Z^3 && Z,\tilde{Z},\sA_Z\mapsto \ainv^2\sA_{Z1}^3 \\
	& Z_1,\sA_{Z1}\mapsto \sA_Z^3 & & Z_1,\sA_{Z1}\mapsto \ainv^2\sA_{Z1}^3 && Z_1,\sA_{Z1}\mapsto \ainv\sA_{Z1}^3\\ \\
	\gamma_{20}\colon & X\mapsto \ainv^2\binv X & \gamma_{21}\colon &X\mapsto \ainv^2\binv X  & \gamma_{22}\colon &X\mapsto \ainv^2\binv X  \\
	& \tilde{X}\mapsto X_b\ainv^2\binv X && \tilde{X}\mapsto X_b\ainv^2\binv X && \tilde{X}\mapsto X_b\ainv^2\binv X \\
	& \tilde{Y}, Y, \sa_Y\mapsto\sa_Y^3& & \tilde{Y}, Y, \sa_Y\mapsto a\sa_Y^3 && \tilde{Y}, Y, \sa_Y\mapsto a^2\sa_Y^3\\
	& Z,\tilde{Z},\sA_Z\mapsto \ainv^2\sA_Z^3& & Z,\tilde{Z},\sA_Z\mapsto \ainv\sA_Z^3 && Z,\tilde{Z},\sA_Z\mapsto\sA_Z^3\\
	& Z_1,\sA_{Z1}\mapsto \ainv\sA^3_Z& & Z_1,\sA_{Z1}\mapsto \sA_Z^3 && Z_1,\sA_{Z1}\mapsto \ainv^2\sA_{Z1}^3\\
	& & &  && \\
\end{align*}
\begin{align*}
	\delta_{00}\colon& \tilde{X},X,\sA_X\mapsto \sA_X^3 & 	\delta_{01}\colon& \tilde{X},X,\sA_X\mapsto \sA_X^3 & 	\delta_{02}\colon& \tilde{X},X,\sA_X\mapsto \sA_X^3 \\
	&  \tilde{Y},Y,\sa_Y\mapsto \sa_Y^3 & & \tilde{Y},Y,\sa_Y\mapsto a\sa_Y^3 & & \tilde{Y},Y,\sa_Y\mapsto a^2\sa_Y^3 \\
	& \tilde{Z},Z,\sA_Z\mapsto\sA_Z^3 & & \tilde{Z},Z,\sA_Z\mapsto \ainv^2\sA_{Z1}^3 & & \tilde{Z},Z,\sA_Z\mapsto \ainv \sA_{Z1}^3 \\
	& Z_1,\sA_{Z1}\mapsto \ainv^2\sA_{Z1}^3 & & Z_1,\sA_{Z1}\mapsto \ainv\sA_{Z1}^3 & & Z_1,\sA_{Z1}\mapsto \sA_{Z1}^3\\ \\
	\delta_{10}\colon& \tilde{X},X,\sA_X\mapsto \ainv\sA_X^3 & 	\delta_{11}\colon& \tilde{X},X,\sA_X\mapsto \ainv\sA_X^3 & 	\delta_{12}\colon& \tilde{X},X,\sA_X\mapsto \ainv\sA_X^3 \\
	& \tilde{Y},Y,\sa_Y\mapsto \sa_Y^3 & & \tilde{Y},Y,\sa_Y\mapsto a\sa_Y^3 & & \tilde{Y},Y,\sa_Y\mapsto a^2\sa_Y^3 \\
	& \tilde{Z},Z,\sA_Z\mapsto \ainv\sA_Z^3 & & \tilde{Z},Z,\sA_Z\mapsto \sA_Z^3 & & \tilde{Z},Z,\sA_Z\mapsto \ainv^2\sA_{Z1}^3\\
	& Z_1,\sA_{Z1}\mapsto \sA_Z^3 & & Z_1,\sA_{Z1}\mapsto \ainv^2\sA_{Z1}^3  & & Z_1,\sA_{Z1}\mapsto \ainv\sA_{Z1}^3 \\ \\	
	\delta_{20}\colon& \tilde{X},X,\sA_X\mapsto \ainv^2\sA_X^3 & 	\delta_{21}\colon& \tilde{X},X,\sA_X\mapsto \ainv^2\sA_X^3 & 	\delta_{22}\colon& \tilde{X},X,\sA_X\mapsto \ainv^2\sA_X^3 \\
	& \tilde{Y},Y,\sa_Y\mapsto \sa_Y^3 & & \tilde{Y},Y,\sa_Y\mapsto a\sa_Y^3& & \tilde{Y},Y,\sa_Y\mapsto a^2\sa_Y^3 \\
	& \tilde{Z},Z,\sA_Z\mapsto \ainv^2\sA_Z^3 & & \tilde{Z},Z,\sA_Z\mapsto \ainv \sA_Z^3& & \tilde{Z},Z,\sA_Z\mapsto \sA_Z^3\\
	& Z_1,\sA_{Z1}\mapsto \ainv\sA_Z^3 & &  Z_1,\sA_{Z1}\mapsto \sA_Z^3 & &  Z_1,\sA_{Z1}\mapsto \ainv^2\sA_{Z1}^3\\
\end{align*}

\end{document}

%% file: FullFSAv4.pdf_tex
\begingroup%
  \makeatletter%
  \providecommand\color[2][]{%
    \errmessage{(Inkscape) Color is used for the text in Inkscape, but the package 'color.sty' is not loaded}%
    \renewcommand\color[2][]{}%
  }%
  \providecommand\transparent[1]{%
    \errmessage{(Inkscape) Transparency is used (non-zero) for the text in Inkscape, but the package 'transparent.sty' is not loaded}%
    \renewcommand\transparent[1]{}%
  }%
  \providecommand\rotatebox[2]{#2}%
  \newcommand*\fsize{\dimexpr\f@size pt\relax}%
  \newcommand*\lineheight[1]{\fontsize{\fsize}{#1\fsize}\selectfont}%
  \ifx\svgwidth\undefined%
    \setlength{\unitlength}{558.52368452bp}%
    \ifx\svgscale\undefined%
      \relax%
    \else%
      \setlength{\unitlength}{\unitlength * \real{\svgscale}}%
    \fi%
  \else%
    \setlength{\unitlength}{\svgwidth}%
  \fi%
  \global\let\svgwidth\undefined%
  \global\let\svgscale\undefined%
  \makeatother%
  \begin{picture}(1,1.38410852)%
    \lineheight{1}%
    \setlength\tabcolsep{0pt}%
    \put(0,0){\includegraphics[width=\unitlength,page=1]{FullFSAv4.pdf}}%
    \put(0.25325577,1.3703306){\makebox(0,0)[lt]{\lineheight{1.25}\smash{\begin{tabular}[t]{l}$\sigma$\end{tabular}}}}%
    \put(0.60307424,1.35208311){\makebox(0,0)[lt]{\lineheight{1.25}\smash{\begin{tabular}[t]{l}$\sigma$\end{tabular}}}}%
    \put(0.24654165,1.18434923){\makebox(0,0)[lt]{\lineheight{1.25}\smash{\begin{tabular}[t]{l}$\rho_{XY}$\end{tabular}}}}%
    \put(0.1028236,1.06861137){\makebox(0,0)[lt]{\lineheight{1.25}\smash{\begin{tabular}[t]{l}$\rho_{XY}$\end{tabular}}}}%
    \put(0.66211054,1.18342297){\makebox(0,0)[lt]{\lineheight{1.25}\smash{\begin{tabular}[t]{l}$\rho_Y$\end{tabular}}}}%
    \put(0.5171264,1.07116151){\makebox(0,0)[lt]{\lineheight{1.25}\smash{\begin{tabular}[t]{l}$\rho_Y$\end{tabular}}}}%
    \put(0.48825567,1.20007279){\makebox(0,0)[lt]{\lineheight{1.25}\smash{\begin{tabular}[t]{l}$\rho_Y$\end{tabular}}}}%
    \put(0.68833668,1.07116153){\makebox(0,0)[lt]{\lineheight{1.25}\smash{\begin{tabular}[t]{l}$\rho_Y$\end{tabular}}}}%
    \put(0.90117445,0.76164018){\makebox(0,0)[lt]{\lineheight{1.25}\smash{\begin{tabular}[t]{l}$\rho_Y$\end{tabular}}}}%
    \put(0.70915049,0.88115168){\makebox(0,0)[lt]{\lineheight{1.25}\smash{\begin{tabular}[t]{l}$\rho_Y$\end{tabular}}}}%
    \put(0.73600698,0.76164018){\makebox(0,0)[lt]{\lineheight{1.25}\smash{\begin{tabular}[t]{l}$\rho_Y$\end{tabular}}}}%
    \put(0.24318994,0.56894462){\makebox(0,0)[lt]{\lineheight{1.25}\smash{\begin{tabular}[t]{l}$\rho_X$\end{tabular}}}}%
    \put(0.46278829,0.26613747){\makebox(0,0)[lt]{\lineheight{1.25}\smash{\begin{tabular}[t]{l}$\rho_X$\end{tabular}}}}%
    \put(0.47823078,0.15132589){\makebox(0,0)[lt]{\lineheight{1.25}\smash{\begin{tabular}[t]{l}$\rho_X$\end{tabular}}}}%
    \put(0.29224942,0.27486583){\makebox(0,0)[lt]{\lineheight{1.25}\smash{\begin{tabular}[t]{l}$\rho_X$\end{tabular}}}}%
    \put(0.26606431,0.45211883){\makebox(0,0)[lt]{\lineheight{1.25}\smash{\begin{tabular}[t]{l}$\rho_X$\end{tabular}}}}%
    \put(0.07605447,0.57364456){\makebox(0,0)[lt]{\lineheight{1.25}\smash{\begin{tabular}[t]{l}$\rho_X$\end{tabular}}}}%
    \put(0.31172039,0.15132589){\makebox(0,0)[lt]{\lineheight{1.25}\smash{\begin{tabular}[t]{l}$\rho_X$\end{tabular}}}}%
    \put(0.2912423,0.10869113){\makebox(0,0)[lt]{\lineheight{1.25}\smash{\begin{tabular}[t]{l}$\varepsilon$\end{tabular}}}}%
    \put(0.09314757,1.26374421){\makebox(0,0)[lt]{\lineheight{1.25}\smash{\begin{tabular}[t]{l}$\varepsilon$\end{tabular}}}}%
    \put(0.50908786,1.26206569){\makebox(0,0)[lt]{\lineheight{1.25}\smash{\begin{tabular}[t]{l}$\varepsilon$\end{tabular}}}}%
    \put(0.45741699,0.02980012){\makebox(0,0)[lt]{\lineheight{1.25}\smash{\begin{tabular}[t]{l}$\tau$\end{tabular}}}}%
    \put(0.19229716,0.01704328){\makebox(0,0)[lt]{\lineheight{1.25}\smash{\begin{tabular}[t]{l}$\mu$\end{tabular}}}}%
    \put(0.71264215,0.11439819){\makebox(0,0)[lt]{\lineheight{1.25}\smash{\begin{tabular}[t]{l}$\tau$\end{tabular}}}}%
    \put(0.10156815,0.45077597){\makebox(0,0)[lt]{\lineheight{1.25}\smash{\begin{tabular}[t]{l}$\rho_X$\end{tabular}}}}%
    \put(0.88237497,0.87443754){\makebox(0,0)[lt]{\lineheight{1.25}\smash{\begin{tabular}[t]{l}$\rho_Y$\end{tabular}}}}%
    \put(0.46001526,0.88195859){\makebox(0,0)[lt]{\lineheight{1.25}\smash{\begin{tabular}[t]{l}$\rho_{XY}$\end{tabular}}}}%
    \put(0.47008647,0.76446134){\makebox(0,0)[lt]{\lineheight{1.25}\smash{\begin{tabular}[t]{l}$\rho_{XY}$\end{tabular}}}}%
    \put(0.29417628,0.88397283){\makebox(0,0)[lt]{\lineheight{1.25}\smash{\begin{tabular}[t]{l}$\rho_{XY}$\end{tabular}}}}%
    \put(0.25187726,1.07263986){\makebox(0,0)[lt]{\lineheight{1.25}\smash{\begin{tabular}[t]{l}$\rho_{XY}$\end{tabular}}}}%
    \put(0.07126725,1.20423678){\makebox(0,0)[lt]{\lineheight{1.25}\smash{\begin{tabular}[t]{l}$\rho_{XY}$\end{tabular}}}}%
    \put(0.31901855,0.76311851){\makebox(0,0)[lt]{\lineheight{1.25}\smash{\begin{tabular}[t]{l}$\rho_{XY}$\end{tabular}}}}%
    \put(0.00854954,1.01867854){\makebox(0,0)[lt]{\lineheight{1.25}\smash{\begin{tabular}[t]{l}$\alpha_{ij}$\end{tabular}}}}%
    \put(0.42368466,1.02006944){\makebox(0,0)[lt]{\lineheight{1.25}\smash{\begin{tabular}[t]{l}$\alpha_{ij}$\end{tabular}}}}%
    \put(0.00640155,0.40169818){\makebox(0,0)[lt]{\lineheight{1.25}\smash{\begin{tabular}[t]{l}$\beta_{ij}$\end{tabular}}}}%
    \put(0.22354154,0.71356946){\makebox(0,0)[lt]{\lineheight{1.25}\smash{\begin{tabular}[t]{l}$\gamma_{ij}$\end{tabular}}}}%
    \put(0.64267194,0.71222666){\makebox(0,0)[lt]{\lineheight{1.25}\smash{\begin{tabular}[t]{l}$\gamma_{ij}$\end{tabular}}}}%
    \put(0.21699617,0.1059408){\makebox(0,0)[lt]{\lineheight{1.25}\smash{\begin{tabular}[t]{l}$\delta_{ij}$\end{tabular}}}}%
    \put(0.63467527,0.10325513){\makebox(0,0)[lt]{\lineheight{1.25}\smash{\begin{tabular}[t]{l}$\delta_{ij}$\end{tabular}}}}%
    \put(0.42160834,0.40169813){\makebox(0,0)[lt]{\lineheight{1.25}\smash{\begin{tabular}[t]{l}$\beta_{ij}$\end{tabular}}}}%
    \put(0,0){\includegraphics[width=\unitlength,page=2]{FullFSAv4.pdf}}%
  \end{picture}%
\endgroup%